\documentclass[11pt, a4paper]{amsart}
\usepackage{amsthm, amssymb, mathtools, cite, graphicx, color, subcaption, amsmath}
\usepackage{stmaryrd}
\usepackage{xypic}
\usepackage{tikz}
\usepackage[utf8]{inputenc}
\usepackage{tikz-cd}
\usepackage{mathtools}
\usepackage{rotating}
\usepackage{lipsum}
\usepackage{marginnote}
\usepackage{hyperref}
\hypersetup{colorlinks}

\usepackage{color} 

\definecolor{darkred}{rgb}{1,0,0} 
\definecolor{darkgreen}{rgb}{0,1,0}
\definecolor{darkblue}{rgb}{0,0,1}

\hypersetup{colorlinks,
linkcolor=darkblue,
filecolor=darkgreen,
urlcolor=darkred,
citecolor=darkgreen}

\newcommand{\rG}{{\mathrm {G}}}
\newcommand{\rE}{{\mathrm {E}}}
\newcommand{\rF}{{\mathrm {F}}}
\newcommand{\rH}{{\mathrm {H}}}
\newcommand{\rS}{{\mathrm {S}}}
\newcommand{\rC}{{\mathrm {C}}}
\newcommand{\rP}{{\mathrm {P}}}
\newcommand{\rB}{{\mathrm {B}}}
\newcommand{\rU}{{\mathrm {U}}}
\newcommand{\rL}{{\mathrm {L}}}
\newcommand{\rSU}{{\mathrm {SU}}}
\newcommand{\rSL}{{\mathrm {SL}}}
\newcommand{\rPSL}{{\mathrm {PSL}}}
\newcommand{\rPGL}{{\mathrm {PGL}}}
\newcommand{\rGL}{{\mathrm {GL}}}
\newcommand{\rSO}{{\mathrm {SO}}}

\newcommand{\rO}{{\mathrm {O}}}
\newcommand{\rSp}{{\mathrm {Sp}}}

\newcommand{\fsl}{{\mathfrak {sl}}}

\newcommand{\fso}{{\mathfrak {so}}}

\newcommand{\fg}{{\mathfrak g}}
\newcommand{\fb}{{\mathfrak b}}
\newcommand{\fh}{{\mathfrak h}}

\newcommand{\fs}{{\mathfrak s}}
\newcommand{\fc}{{\mathfrak c}}

\newcommand{\fp}{{\mathfrak p}}
\newcommand{\fl}{{\mathfrak l}}

\newcommand{\fu}{{\mathfrak u}}

\newcommand{\pos}{{\mathrm {pos}}}

\newcommand{\Sym}{{\mathrm {Sym}}}

\newcommand{\Hom}{{\mathrm{Hom}}}

\newcommand{\End}{{\mathrm{End}}}

\newcommand{\Ad}{{\mathrm{Ad}}}
\newcommand{\Tr}{{\mathrm{Tr}}}
\newcommand{\Id}{{\mathrm{Id}}}
\newcommand{\ad}{{\mathrm{ad}}}
\newcommand{\rk}{{\mathrm{rk}}}

\newcommand{\hook}{\hook}

\newcommand{\Op}{\mathfrak{Op}}

\newcommand{\Bb}{{\mathcal B}}

\newcommand{\Xx}{{\mathcal X}}

\newcommand{\Ll}{{\mathcal L}}

\newcommand{\Ee}{{\mathcal E}}
\newcommand{\Ff}{{\mathcal F}}

\newcommand{\Kk}{{\mathcal K}}
\newcommand{\Mm}{{\mathcal M}}

\newcommand{\Oo}{{\mathcal O}}

\newcommand{\Qq}{{\mathcal Q}}

\newcommand{\Vv}{{\mathcal V}}

\newcommand{\Tt}{{\mathcal T}}
\newcommand{\Ww}{{\mathcal W}}
\newcommand{\Zz}{{\mathcal Z}}

\def    \C      {{\mathbb C}}
\def    \R      {{\mathbb R}}
\def    \Z      {{\mathbb Z}}

\def    \ab     {{\alpha\beta}}

\usetikzlibrary{matrix}
\theoremstyle{plain}
\newtheorem{theorem}{Theorem}[section]
\newtheorem*{theorem*}{Theorem}
\newtheorem{lemma}[theorem]{Lemma}
\newtheorem{proposition}[theorem]{Proposition}

\newtheorem*{corollary*}{Corollary}

\numberwithin{equation}{section}
\theoremstyle{definition}
\newtheorem{definition}[theorem]{Definition}
\theoremstyle{remark}
\newtheorem{remark}[theorem]{Remark}
\newtheorem{example}[theorem]{Example}

\title[(G,P)-Opers and Slodowy slices]{(G,P)-Opers and global Slodowy slices}
\author[B. Collier]{Brian Collier}
\address{B. Collier, Mathematical Sciences Research Institute}
\email{briancollier01@gmail.com}

\author[A. Sanders]{Andrew Sanders}
\address{A. Sanders, Department of Mathematics, Heidelberg University}
\email{asanders@mathi.uni-heidelberg.de}
\thanks{B.C.'s research leading to these results has received funding from a National Science Foundation Mathematical Sciences Postdoctoral Fellowship, NSF MSPRF no. 1604263. The authors also acknowledges support from U.S. National Science Foundation grants DMS 1107452, 1107263, 1107367 "RNMS: GEometric structures And Representation varieties" (the GEAR Network).}
\begin{document}


\date{\today}

\begin{abstract}
In this paper, we introduce a generalization of $\rG$-opers for arbitrary parabolic subgroups $\rP<\rG$
of a complex semisimple Lie group. For parabolic subgroups associated to ``even nilpotents'', we parameterize $(\rG,\rP)$-opers by an object generalizing the base of the Hitchin fibration. In particular, we describe and parameterize families of opers associated to  higher Teichm\"uller spaces.
\end{abstract}

\maketitle
\tableofcontents

\section{Introduction}\label{sec: intro}

Given a second order holomorphic ordinary differential equation on the upper half plane $\mathbb{H}\subset \mathbb{C},$ if there exists a basis of holomorphic solutions $\{u_{1}, u_{2}\}$ which never simultaneously vanish, they define a holomorphic map $D: \mathbb{H}\rightarrow \mathbb{CP}^{1}$ by the formula $D=u_{1}/u_{2}.$  Imposing certain invariance and non-degeneracy conditions on the ODE under the action of a co-compact Fuchsian group $\Gamma<\textnormal{Aut}(\mathbb{H})$ leads to a representation $\rho: \pi_{1}(\mathbb{H}/\Gamma)\rightarrow \rPGL_{2}\C$ for which the map $D$ is an equivariant holomorphic immersion.

Pairs $(D, \rho)$ arising in this way are called complex projective structures on the Riemann surface $X=\mathbb{H}/\Gamma$ (see \cite{GUN66}).  
Treating the ODE itself as fundamental leads to a differential geometric description of such structures as particular rank-$2$ holomorphic flat vector bundles on $X$.
Extensions of complex projective structures to higher order equations have been under development since (at least) the 1960's \cite{TEL60,HEJ75}, with the relevant objects now being certain rank-$n$ holomorphic flat vector bundles on $X.$ 

Generalizing the case of $\rP\rGL_n\C$, Drinfeld-Sokolov \cite{DS84} developed a local theory of ordinary holomorphic differential equations which takes as input an arbitrary complex reductive Lie group $\rG$. 
In their work on the geometric Langlands conjecture, Beilinson-Drinfeld \cite{BeilinsonDrinfeldOPERS} globalized this theory to the notion of a $\rG$-\emph{oper} on a Riemann surface $X$.
According to Beilinson-Drinfeld, a $\rG$-oper on $X$ is a triple $(E_{\rG}, E_{\rB}, \omega)$, where $E_{\rG}$ is a holomorphic principal $\rG$-bundle on $X$, $E_{\rB}$ is a holomorphic reduction to a Borel subgroup $\rB<\rG,$ and $\omega$ is a holomorphic connection on $E_{\rG}$ which satisfies certain transversality and nowhere vanishing properties with respect to the reduction $E_{\rB}.$ 

Around the same time, Hitchin \cite{liegroupsteichmuller} discovered a remarkable component of the
space of conjugacy classes of homomorphisms $\pi_{1}(X)\rightarrow \rG_{\mathbb{R}},$ where $\rG_{\mathbb{R}}<\rG$ is a split real form of a complex simple Lie group (e.g. $\rPSL_n\R$ is the split real form of $\rPSL_n\C$) and $X$ is a compact Riemann surface of genus $g\geq 2$. This component is now called the \emph{Hitchin} component. When $\rG=\rPSL_2\C$, the Hitchin component is the \emph{Fuchsian} space of conjugacy classes of holonomy representations $\rho: \pi_{1}(X)\rightarrow \rPSL_{2}\R$ of hyperbolic structures on the underlying oriented topological surface $\Sigma$ and, via the uniformization theorem, is identified with the Teichm\"{u}ller space of $\Sigma.$ 

In \cite{AnosovFlowsLabourie}, Labourie proved that the representations in Hitchin components generalize many geometric features of Fuchsian space: in particular, they are discrete and faithful quasi-isometric embeddings. This was achieved by showing that Hitchin representations satisfy a very strong dynamical property, called the \emph{Anosov} property. Since Labourie's work, there has been intense activity aimed at organizing and classifying Anosov representations.  

Although Hitchin representations and $\rG$-opers are very different geometrically, the resulting moduli spaces share many common features. Most notably, both moduli spaces are parameterized by a vector space of holomorphic differentials on $X$ called the base of the Hitchin fibration \cite{liegroupsteichmuller,BeilinsonDrinfeldOPERS}. More recently, in \cite{GaiottoLimitsOPERS}, the authors showed that an operation called the \emph{conformal limit} identifies these parameterizations, proving a conjecture of Gaiotto \cite{GaiottoConj}. 

In the theories of $\rG$-opers and Hitchin representations, a primary role is played by a Borel subgroup.  
In general, Anosov representations are defined relative to a parabolic subgroup.  
Over the last fifteen years, other components of surface group representations into real Lie groups have been discovered, which in many ways are analogues of Hitchin components, but with a larger parabolic subgroup replacing the role of the Borel subgroup. Such components are now referred to as higher Teichm\"uller components.

With these motivations, in this paper we introduce the notion of a $(\rG,\rP)$-oper for $\rP<\rG$ an arbitrary parabolic subgroup. Our main results generalize the parameterization of Hitchin components and $\rG$-opers by the Hitchin base.  While we cast the general theory in Lie theoretic terms, in most cases the objects can be made relatively explicit with matrices and vector bundles. A significant part of the paper is dedicated to explicating this aspect.   



\subsection*{(G,P)-opers}
Let $X$ be a compact Riemann surface of genus $g\geq 2$ and let $\rP<\rG$ be a parabolic subgroup of a connected complex semisimple Lie group $\rG$. 
Roughly, a \emph{$(\rG,\rP)$-oper on $X$} is a triple $(E_{\rG}, E_{\rP}, \omega)$, where $E_{\rG}$ is a holomorphic principal $\rG$-bundle on $X$, $E_{\rP}$ is a holomorphic reduction to the parabolic $\rP<\rG,$ and $\omega$ is a holomorphic connection on $E_{\rG}$ which satisfies certain transversality and nowhere vanishing properties with respect to the reduction $E_{\rP}$. The subtle point is defining the correct notion of nowhere vanishing (see Definition \ref{Def oper def}). 

A simple example is given by $\rG=\rSL_{2n}\C$ and $\rP$ is the stabilizer of a subspace $\C^n\subset\C^{2n}.$ In this case, a $(\rG,\rP)$-oper is equivalently defined as a $4$-tuple $(\Vv_{2n},\Vv_n,\nabla, \Omega)$, where $\Vv_{2n}$ is a rank $2n$ holomorphic vector bundle with a holomorphic volume form $\Omega$, $\Vv_n\subset\Vv_{2n}$ is a rank $n$ holomorphic sub-bundle, and $\nabla$ is a holomorphic connection on $\Vv_{2n}$ for which $\Omega$ is parallel and such that the induced $\Oo_X$-linear map $\nabla:\Vv_n\to\Kk\otimes \Vv_{2n}/\Vv_n$ is an isomorphism. Here, $\Kk$ is the holomorphic cotangent bundle of $X$.

We briefly describe the picture we wish to generalize. In \cite{selfduality}, Hitchin introduced the (coarse) moduli space $\Mm^0_X(\rG)$ of poly-stable $\rG$-Higgs bundles on $X.$  Some relevant aspects of the theory of Higgs bundles and opers are: 
\begin{enumerate}
      \item  There is a homeomorphism $\Tt:\Mm^0_X(\rG)\to\Mm_X^1(\rG)$ from the moduli space of poly-stable Higgs bundles to the moduli space $\Mm_X^1(\rG)$ of reductive holomorphic flat $\rG$-bundles on $X$ \cite{selfduality,harmoicmetric,canonicalmetrics,SimpsonVHS}. This is known as the non-abelian Hodge correspondence.
      \item There is a holomorphic map $\Mm^0_X(\rG)\to\bigoplus_{j=1}^{\rk\rG}\rH^0(X, \Kk^{m_j+1})$ called the \emph{Hitchin} fibration \cite{IntSystemFibration}. Here $\{m_j\}$ are the exponents of $\fg$. 
      \item The Hitchin fibration has a section $s_h$ and the Hitchin components are the holonomy representations of $\Tt\circ s_h$ \cite{liegroupsteichmuller}.
      \item Each connected component of isomorphism classes of $\rG$-opers is parameterized by $\bigoplus_{j=1}^{\rk\rG}\rH^0(X,\Kk^{m_j+1})$ \cite{BeilinsonDrinfeldOPERS}.
  \end{enumerate} 
We summarize this in the following diagram:
\[\xymatrix@R=1.5em{\Mm^0_X(\rG)\ar[r]^\Tt\ar[d]&\Mm^1_X(\rG)\\
\bigoplus_{j=1}^{\rk\rG}\rH^0(X,\Kk^{m_j+1})\ar@<.3pc>@/^.8pc/[u]^{s_h}\ar[ur]_{op}}.\]

In fact, for $\lambda\in\C,$ the above moduli spaces fit into a family of moduli spaces $\Mm^\lambda_X(\rG)$ of holomorphic $\lambda$-connections. Moreover, there is a family of maps $op_\lambda:\bigoplus_{j=1}^{\rk\rG}\rH^0(X,\Kk^{m_j+1})\to\Mm^\lambda_X(\rG)$ with $op_0=s_H$ and $op_1=op.$ Roughly, the maps $op_\lambda$ are defined by choosing a base point $\rP\rSL_2\C$-oper and adding holomorphic differentials to the connection as coefficients in certain highest weight spaces of $\fsl_2\C$-representations.  The main goal of this paper is to extend this story to the setting of $(\rG, \rP)$-opers, with the fundamental new idea being the identification of the analogue of the Hitchin base.

Our main theorems parameterizing $(\rG, \rP)$-opers apply to a class of parabolics called \emph{even Jacobson-Morozov parabolics} (see Definition \ref{def: even JM parabolic}). Such parabolics are associated to embeddings of $\rPSL_2\C$ in the adjoint Lie group, or equivalently, even $\fsl_2\C$-subalgebras of the Lie algebra $\fg.$ There are many examples, for instance, the Borel subgroup $\rB<\rG$ and the stabilizer of a linear subspace $\C^n\subset\C^{2n}$. 
The object which parameterizes such $(\rG,\rP)$-opers (i.e., the generalization of the Hitchin base) is a groupoid which we call the \emph{Slodowy category}. A key player in the definition of the Slodowy category is the centralizer of the associated $\fsl_2\C$-subalgebra. 

Given an even $\fsl_2\C$-subalgebra $\fs\subset\fg$, let $\rS<\rG$ be the connected Lie group with Lie algebra $\fs$ and $\rC<\rG$ be the centralizer of $\rS.$ The $\fs$-Slodowy category $\Bb_{\fs, X}(\rG)$ has objects $(E_\rC,\psi_0,\psi_{m_1},\cdots,\psi_{m_N})$, where $E_\rC$ is a holomorphic $\rC$-bundle with holomorphic connection $\psi_0$ and, for $1\leq j\leq N,$ $\psi_{m_j}$ is a holomorphic section of an associated vector bundle $E_\rC[V_{2m_j}]\otimes \Kk^{m_j+1}$.
 Here the vector spaces $V_{2m_j}$ are highest weight spaces of certain $\fsl_2\C$ representations (see Definition \ref{def s Slodowy category}).
When $\rB<\rG$ is the Borel subgroup, $\rC$ is the center of $\rG$ and $\Bb_{\fs, X}(\rG)$ is the product of the Hitchin base with the finite set of $\rC$-bundles on $X.$ When $\rP<\rSL_{2n}\C$ is the stabilizer of a subspace $\C^n\subset\C^{2n}$, objects of the Slodowy category correspond to triples $(\Ww,\psi_0,\psi_1),$ where $\Ww$ is a rank $n$ holomorphic vector bundle on $X$ with $\det(\Ww)^{\otimes 2}$ trivial, $\psi_0$ is a holomorphic connection on $\Ww$ compatible with the trivialization of $\det(\Ww)^{\otimes 2},$ and $\psi_1\in \rH^0(X, \Kk^2\otimes \End(\Ww)).$

\begin{remark}
In the case of the Borel subgroup $\rB<\rG$, the work of Hitchin and Beilinson-Drinfeld yield holomorphic parameterizations of the moduli spaces of Hitchin representations and $(\rG ,\rB)$-opers, respectively. The existence of these moduli spaces is simplified by the subgroup $\rC$ in the previous paragraph being discrete.
In general, the construction of relevant moduli spaces will depend on natural stability conditions. To avoid complications with stability conditions, we work entirely with groupoids/categories and leave the topic of constructing moduli spaces/stacks for another time.
\end{remark}


\subsection*{Main results}
 
Denote the category of $(\rG,\rP)$-opers on $X$ by $\Op_X(\rG,\rP).$ For an even $\fsl_2\C$ subalgebra $\fs\subset\fg$, let $\rS<\rG$ be the associated connected Lie group and $\rP<\rG$ be the associated even Jacobson-Morozov parabolic. In Theorem \ref{thm: equivalence of categories} we establish the following: For each $\rS$-oper $\Theta$ on $X,$ there is a functor
 \[F_\Theta:\Bb_{\fs, X}(\rG)\to\Op_X(\rG,\rP),\] 
which is an equivalence of categories (see also Theorem \ref{thm: parameterization picking PSL2 oper}).  In particular, the  functor $F_\Theta$ induces a bijection at the level of isomorphism classes.
The functor $F_\Theta$ is defined in Definition \ref{def Slodowy functor}. Roughly, the sections $\psi_{m_j}$ are \emph{coefficients} in certain highest weight spaces of $\fg$ as an $\fsl_{2}\C$-module, and are used to affinely deform the fixed $\rS$-oper as a $(\rG, \rP)$-oper. This is analogous to the Hitchin section when $\rP$ is a Borel subgroup $\rB<\rG$.

To remove the choice of $\rS$-oper, we define a full subcategory $\widehat\Bb_{\fs, X}(\rG)$ of $\Bb_{\fs, X}(\rG)$ which is analogous to setting the quadratic differential in the Hitchin base to be zero. This yields a functor $F:\Op_X(\rS,\rB_\rS)\times\widehat\Bb_{\fs,X}(\rG)\to\Op_X(\rG,\rP)$ given by $F(\Theta,\Xi)=F_\Theta(\Xi),$ which we call the \emph{$\fs$-Slodowy functor}.
 
\begin{theorem} (Theorem \ref{thm: no quadratic equivalence})
 Let $X$ be a compact Riemann surface of genus $g\geq 2$ and $\rG$ be a connected complex semisimple Lie group. Let $\fs\subset\fg$ be an even $\fsl_2\C$-subalgebra, $\rS<\rG$ be the associated connected subgroup and $\rP<\rG$ be the associated even Jacobson-Morozov parabolic. Then the $\fs$-Slodowy functor 
 \begin{align*}
 F:\Op_X(\rS,\rB_\rS)\times \widehat\Bb_{\fs, X}(\rG)&\to\Op_{X}(\rG,\rP) \\
  (\Theta,\Xi)&\mapsto F_\Theta(\Xi)
 \end{align*}
 is an equivalence of categories when $\rS\cong\rPSL_2\C$ and essentially surjective and full when $\rS\cong\rSL_2\C.$ 
 In particular, the Slodowy functor induces a bijection on isomorphism classes. 
\end{theorem}
\begin{remark}
	For $\rG=\rSL_{kr}\C$ and the parabolic $\rP$ stabilizing a partial flag $\C^r\subset\C^{2r}\subset\cdots\subset\C^{kr}$, the above theorem recovers results of Biswas in \cite{BISWASopersn=kr}. Partially motivated by the work presented in the current article, Biswas-Schaposnik-Wang recently considered certain symplectic and orthogonal analogues of such partial flags, see \cite{BLY_B-opers}. 
\end{remark}

The category of $(\rG,\rP)$-opers and the Slodowy category have natural $\lambda$-connection generalizations, denoted  respectively by $\Op_X^\lambda(\rG,\rP)$ and $\rB_{\fs,X}^\lambda(\rG)$. The Slodowy functor also generalizes to this context, and in Theorem \ref{Thm lambda equiv} we establish the analogue of the above theorem in this setting. 
When $\lambda=0$, a $(\lambda,\rG,\rP)$-oper is a special type of Higgs bundle. 

For $\rG=\rSL_{2n}\C$ and $\rP$ the stabilizer of a linear subspace $\C^n\subset\C^{2n}$, the objects of the Slodowy category $\Bb_{\fs,X}^0(\rG)$ consist of triples $(\Ww,\psi_0,\psi_1)$, where $\Ww$ is a holomorphic vector bundle of rank $n$ such that $\det(\Ww)^{\otimes 2}$ is trivial, $\psi_0\in \rH^0(X,\Kk\otimes \End(\Ww))$ is a $\Kk$-twisted traceless endomorphism, and $\psi_1\in \rH^0(X,\Kk^2\otimes \End(\Ww))$ is a $\Kk^2$-twisted endomorphism.  

For the $(0,\rSL_2\C,\rB)$-oper $\Theta=(\Kk^\frac{1}{2}\oplus \Kk^{-\frac{1}{2}}, \left(\begin{smallmatrix}
	0&q_2\\1&0
\end{smallmatrix}\right) )$, the functor $F_\Theta$ is defined by
\begin{equation*}
	\label{eq max sunn}F_\Theta(\Ww,\psi_0,\psi_1)=\left((\Ww\otimes \Kk^{\frac{1}{2}})\oplus (\Ww\otimes \Kk^{-\frac{1}{2}}), \begin{pmatrix}
	\psi_0&q_2\otimes\Id_\Ww+\psi_1\\\Id_\Ww&\psi_0
\end{pmatrix}\right).
\end{equation*}
When $\psi_0=0$ and $q_2=0,$ this is an $\rSU_{n,n}$-Higgs bundle, and the above description recovers the so called Cayley correspondence for maximal $\rSU_{n,n}$-Higgs bundles of \cite{UpqHiggs}. Moreover, when the associated Higgs bundles are poly-stable, the nonabelian Hodge correspondence identifies them with the higher Teichm\"uller spaces known as \emph{maximal $\rSU_{n,n}$-representations}. 

Due to the work of Guichard-Wienhard \cite{PosRepsGWPROCEEDINGS} on $\Theta$-positive structures, it is reasonable to define a higher Teichm\"uller space as a connected component of the space of conjugacy classes of representations $\pi_1X\to\rG^\R$ consisting entirely of $\Theta$-positive Anosov representations. Here $\rG^\R$ is a real form of a complex simple Lie group which admits a so called $\Theta$-positive structure, see \cite{PosRepsGWPROCEEDINGS}. 
In \S \ref{Section examples} we describe how, ignoring stability conditions, our construction recovers all known Higgs bundle parameterizations of higher Teichm\"uller spaces. In particular, we recover the Hitchin component of \cite{liegroupsteichmuller} when $\rG^\R$ is a split real Lie group, the Cayley correspondence for maximal representations of \cite{BGRmaximalToledo} when $\rG^\R$ is a Hermitian Lie group of tube type, and the generalized Cayley correspondence of \cite{so(pq)BCGGO} when $\rG^\R\cong\rSO_{p,q}$. 
There is exactly one more expected family of higher Teichm\"uller spaces which are associated to certain real forms of the exceptional Lie groups $\rF_4,$ $\rE_6,$ $\rE_7$ and $\rE_8.$ For all expected higher Teichm\"uller spaces, the  parameterization of the associated moduli spaces of Higgs bundles by the Slodowy functor will be described in \cite{MagicalBCGGO}.

  

\begin{remark}
	Higher Teichm\"uller spaces only arise for very special even JM-parabolics. It would be interesting to understand the geometric properties of representations associated to applying nonabelian Hodge to the set of poly-stable $(0,\rG,\rP)$-opers for arbitrary even JM-parabolic.
\end{remark}

Finally, we discuss the relationship between $(\rG,\rP)$-opers, Simpson's partial oper stratification and the conformal limit. In \cite{GaiottoLimitsOPERS}, the authors showed that the space of $(\rG,\rB)$-opers are identified with the Hitchin section by an operation called the \emph{conformal limit}, proving a conjecture of Gaiotto \cite{GaiottoConj}. For $\rG=\rSL_n\C$, the space of $\rB$-opers and the Hitchin section are each a closed stratum of natural stratifications introduced by Simpson \cite{SimpsonDeRhamStrata}. Moreover, under a smoothness assumption, the conformal limit correspondence for arbitrary strata was established in \cite{ConfLimColWent}. In the generality of this paper, it is natural to expect that the conformal limit also identifies $(0,\rG,\rP)$-opers with $(1,\rG,\rP)$-opers. Indeed, for the case $\rG=\rSL_{2n}\C$ and $\rP$ the stabilizer of a subspace $\C^n\subset\C^{2n}$, when the bundle $\Ww$ in the Slodowy category is a stable bundle, this follows from the work in \cite{ConfLimColWent}. 

\vspace{.5cm}

\noindent\textbf{Acknowledgments:} We would like to thank Anna Wienhard for encouraging us to investigate a generalization of opers, and Jeff Adams, François Labourie and Richard Wentworth for many useful conversations and suggestions. We also express our appreciation for the hospitality of the University of Heidelberg and the Simons Center program \emph{Geometry and Physics of Hitchin Systems}, where much of this work was carried out.

\section{Lie theory preliminaries}
In this section, we recall most of the Lie theory which will be used throughout the paper. For this section, $\rG$ will be a connected complex semi-simple Lie group with Lie algebra $\fg$ and Killing form $B_\fg.$ 

\subsection{Parabolics}\label{sec: parabolics}
A Borel subgroup $\rB<\rG$ is a maximal connected solvable subgroup. There is a unique Borel subgroup up to conjugation. A subgroup $\rP<\rG$ is called a parabolic subgroup if it contains a Borel subgroup. 

For a parabolic $\rP<\rG,$ let $\rU<\rP$ be the unipotent radical.
There is a canonical filtration of $\rU=\rU^1,$
\[\rU^m<\cdots<\rU^2<\rU^1<\rP.\]
The quotient $\rL=\rP/\rU$ is a reductive group called the Levi factor of $\rP.$ 
The exact sequence $1\to \rU\to \rP\to \rL \to 1$ is split, and a choice of splitting defines a Levi subgroup of $\rL<\rP.$ 

Let $\fu\subset\fp$ be the Lie subalgebra of the unipotent radical $\rU<\rP$. There is a canonical $\rP$-invariant Lie algebra filtration of $\fg$ defined by 
\begin{equation}
    \label{eq canonical filtration of parabolics}0\subset\fg^{m}\subset\cdots\subset \fg^{1}\subset\fg^0\subset\fg^{-1}\subset\cdots\subset\fg^{-m}=\fg,
\end{equation}
where $\fg^j=\{x\in\fg~|~\ad_x(\fu)\subset\fg^{j+1}\}.$ In particular, $\fg^1=\fu$ and $\fg^0=\fp.$
 
A choice of a Levi subgroup $\rL<\rP$ induces an $\rL$-invariant splitting $\fp=\fl\oplus\fu$, where $\fl$ is the Lie algebra of $\rL.$ Moreover, $\fg$ acquires an $\rL$-invariant grading
\begin{equation}
    \label{eq associated graded Z grading of parabolic}\fg=\bigoplus_{j\in\Z}\fg_j,
\end{equation}
which is $\rL$-equivariantly isomorphic to the associated graded of the canonical filtration 
\eqref{eq canonical filtration of parabolics}. Thus, $\fg_j\cong\fg^j/\fg^{j+1}$ as $\rL$-modules, and $\fg^k\cong\bigoplus_{j\geq k}\fg_j$ as $\rP$-modules. In particular, 
\[\xymatrix{\fp\cong\bigoplus\limits_{j\geq0}\fg_j,&\fl\cong\fg_0&\text{and}&\fu\cong\bigoplus\limits_{j>0}\fg_j}.\]

By a theorem of Vinberg (see \cite[Theorem 10.19]{knappbeyondintro}), each graded piece $\fg_j$ has a unique dense open $\rL$-orbit $\Oo_j\subset\fg_j$.  Since the action of $\rP$ and $\rL$ on $\fg^j/\fg^{j+1}$ are the same, there is a unique open dense $\rP$-orbit in $\fg^j/\fg^{j+1}$. 


\begin{example}\label{ex: SL4 Lie data}
For the group $\rG=\rSL_4(\C)$, the group of unit determinant upper triangular matrices defines a Borel subgroup. The parabolic subgroups which contain this choice of $\rB$ consist of block upper triangular matrices. Geometrically, $\rB$ is the stabilizer of a complete flag in $\C^4$ and other parabolics are stabilizers of partial flags in $\C^4.$ 
Here are some explicit examples:
    \begin{enumerate}
        \item The Borel subgroup $\rP=\rB$ which stabilizes a complete flag $\C\subset\C^2\subset\C^3\subset\C^4$ consists of unit determinant upper triangular matrices. The unipotent radical $\rU<\rB$ consists upper triangular matrices with $1$'s on the diagonal. The diagonal matrices define a Levi subgroup $\rL<\rB$, and, with this choice, the graded pieces $\fg_{j}$ are given by the $j^{th}$-super diagonal. In particular, $\fg_{-1}$ is given by matrices of the form
        \[\left(\begin{smallmatrix}
        0&&&\\x_{1}&0&&\\&x_2&0&\\&&x_3&0
    \end{smallmatrix}\right),\]
    The unique open $\rL$-orbit $\Oo_{-1}\subset\fg_{-1}$ is defined by $x_j\neq0$ for all $j.$

        \item The group $\rP<\rSL_4(\C)$ which stabilizes a partial flag $\C^2\subset\C^4$ consists of unit determinant $(2\times 2)$-block matrices of the form
        \[\left(\begin{matrix}
            A&B\\ 0&D
        \end{matrix}\right).\]
        The unipotent radical $\rU<\rP$ is the subgroup where $A=\Id=D,$ and the subgroup where $B=0$ defines a Levi subgroup. This choice of Levi subgroup determines the grading
        \[\fg=\fg_{-1}\oplus\fg_0\oplus\fg_1=\left(\begin{matrix}
            0&0\\C&0
        \end{matrix}\right)\oplus
        \left(\begin{matrix}
            A&0\\ 0 &D
        \end{matrix}\right)\oplus\left(\begin{matrix}
            0&B\\ 0&0 
        \end{matrix}\right).\]
        The unique open $\rL$-orbit $\Oo_{-1}\subset\fg_{-1}$ is defined by $\det(C)\neq0.$ 

        \item
        The group $\rP<\rSL_4(\C)$ which stabilizes a partial flag $\C^1\subset\C^3\subset\C^4$ is given by unit determinant matrices of the form
        \[\left(\begin{smallmatrix}
            *&*&*&*\\  &*&*&*\\&*&*&*\\&&&*
        \end{smallmatrix}\right).\]
        Block diagonal matrices define a Levi subgroup, and with this choice, 
        \[\fp=\fg_0\oplus\fg_1\oplus\fg_2=\left(\begin{smallmatrix}
            *&0&0&0\\&*&*&0\\&*&*&0\\&&&*
        \end{smallmatrix}\right) \oplus \left(\begin{smallmatrix}
            0&*&*&0\\&0&0&*\\&0&0&*\\&&&0
        \end{smallmatrix}\right)\oplus\left(\begin{smallmatrix}
            0&0&0&*\\&0&0&0\\&0&0&0\\&&&0
        \end{smallmatrix}\right).\]
        The unique open $\rL$-orbit $\Oo_{-1}\subset\fg_{-1}$ is given by 
        \[\Oo_{-1}=\left\{\left(\begin{smallmatrix}
            0&&&\\a&0&0&\\b&0&0&\\0&c&d&0
        \end{smallmatrix}\right)~|~\left( \begin{smallmatrix}
            c&d
        \end{smallmatrix}\right)\left( \begin{smallmatrix}
            a\\b
        \end{smallmatrix}\right)\neq0\right\}.\]

        \item For the parabolic subgroup $\rP<\rSL_4(\C)$ which stabilizes a line $\C\subset\C^4$, block diagonal matrices define a Levi subgroup. With this choice, the associated grading is
        \[\fg=\fg_{-1}\oplus\fg_0\oplus\fg_1=\left(\begin{smallmatrix}
            0&&&\\ *&0&0&0\\ *&0&0&0\\ *&0&0&0
        \end{smallmatrix}\right)\oplus \left(\begin{smallmatrix}
            *&&&\\&*&*&*\\&*&*&*\\&*&*&*
        \end{smallmatrix}\right) \oplus \left(\begin{smallmatrix}
            0&*&*&*\\&0&0&0\\&0&0&0\\&0&0&0
        \end{smallmatrix}\right).\]
        The unique open $\rL$-orbit $\Oo_{-1}\subset\fg_{-1}$ consists of nonzero vectors.
    \end{enumerate}
\end{example}

The next example involves some basic root theory. 

\begin{example}\label{ex: G/B Lie data}
    For $\rP=\rB<\rG$ the Borel subgroup, a Levi subgroup $\rL<\rB$ is a \emph{Cartan} subgroup. The data $(\rL, \rB)$ determines a set of positive simple roots $\{\alpha_1,\cdots,\alpha_{\rk(\fg)}\}.$  The spaces $\fg_j$ consist of direct sums of root spaces $\fg_{\alpha}$ associated to roots $\alpha=\sum_{i=1}^{\rk(\fg)}n_i\alpha_i$ with \emph{length} $\sum n_i=j$. 
    In particular, the $\fg_{-1}$ space consists of the direct sum of negative simple root spaces. In this case, an element 
    \[x=(x_{1},\cdots,x_{\rk(\fg)})\in \fg_{-\alpha_1}\oplus\cdots\oplus\fg_{-\alpha_{\rk(\fg)}}\] is in the open dense $\rL$-orbit $\Oo_{-1}\subset\fg_{-1}$ if and only if $x_{j}\neq0$ for all $j.$
\end{example}

\subsection{Nilpotents and the Jacobson-Morozov Theorem}\label{sec: sl2 subalgebras} 

An element $e\in\fg$ is nilpotent if the linear map $\ad_e:\fg\to\fg$ is a nilpotent endomorphism of $\fg.$  Denote the centralizer of a nilpotent $e\in\fg$ by $V(e):=V$,
\[V=\ker(\ad_e:\fg\to\fg)~.\]

The Jacobson-Morozov theorem defines a bijective correspondence between conjugacy classes of $\fsl_2\C$-subalgebras of $\fg$ and conjugacy classes of nonzero nilpotents in $\fg$ (see for example \cite[\S 3]{CollingwoodMcGovern}). Namely, every non-zero nilpotent $e\in\fg$ can be completed to an $\fsl_2$-triple $\langle f,h,e\rangle\subset\fg,$ where 
\[\xymatrix{[h,e]=2e,&[h,f]=-2f&\text{and}&[e,f]=h,}\]
and if $\langle f,h,e\rangle$, $\langle f',h,e\rangle$ are two such $\fsl_2$-triples then $f=f'.$ Throughout the paper, we will use the letter $\fs$ to denote an $\fsl_2$-subalgebra. 

Fix an $\fsl_2$-triple $\langle f,h,e\rangle=\fs\subset\fg.$  This data decomposes $\fg$ into $\fsl_2\C$-modules and $\ad_h$-weight spaces. Namely, as an $\fsl_2\C$-module,
\begin{equation}
    \label{eq sl2 module decomp}\fg=\bigoplus_{j\geq 0}W_j,
\end{equation}
where $W_j$ is isomorphic to a direct sum of $n_j$-copies (with $n_j\geq0$) of the unique $(j+1)$-dimensional $\fsl_2\C$ representation, and, as $\ad_{h}$-weight spaces,
\begin{equation}
    \label{eq adh wieght decomp}\fg=\bigoplus_{j\in\Z}\fg_j,
\end{equation}
where $\fg_j:=\{x\in\fg~|~\ad_h(x)=jx\}.$
By definition, $e\in\fg_2$ and $f\in\fg_{-2}$. Note also that the trivial representation $W_0$ is a reductive subalgebra since it is the centralizer of $\fs$. We will denote this subalgebra by $W_0:=\fc(\fs):=\fc.$

The subalgebra $\fb_\fs=\langle h,e\rangle\subset\fg$ is a Borel subalgebra of $\fs.$  
The centralizer $V$ of $e$ decomposes $(\fc\oplus\fb_\fs)$-invariantly as 
\begin{equation}
    \label{eq highest weight space decomp}V=\bigoplus\limits_{j\geq 0}V_j,
\end{equation}
where $V _j=\fg_j\cap W_j$ is the highest weight space of $W_j$. The affine subspace 
\[f+V\subset\fg\] 
 is called the \emph{Slodowy slice} through $f.$ It is transverse to the $\rG$-orbit of $f$. 

\begin{definition}
    Fix an $\fsl_2$-triple $\langle f,h,e\rangle$. We call the $(\fb_\fs\oplus\fc)$-invariant decomposition \eqref{eq highest weight space decomp} of the centralizer $V$ of $e$ the \emph{Slodowy data} of $\langle f,h,e\rangle$.
\end{definition}

The following proposition is immediate and will be useful later. 
\begin{proposition}\label{prop: s+b_0 decomp of V and V_2}Let $\bigoplus_jV_j$ be the Slodowy data of an $\fsl_2$-triple $\langle f,h,e\rangle.$ Then the subspace $V_2$ has an $\fc\oplus\fb_\fs$-invariant decomposition
\begin{equation}
    \label{eq V_2 decomposition}V_2=\langle e\rangle\oplus \hat V_2,
\end{equation}
where $\hat V_2$ is the kernel of the linear functional $B_\fg(f,-)\big |_{V_2}$. 
\end{proposition}

The decomposition \eqref{eq adh wieght decomp} associates a parabolic subalgebra $\fp$ to an $\fsl_2$-triple $\langle f,h,e\rangle$: 
\[\fp=\bigoplus\limits_{j\geq 0}\fg_j.\]
Moreover, this parabolic comes with a choice of Levi subalgebra $\fl=\fg_0\subset\fp.$ Note that $\fp$ also decomposes as
\[\fp=V\oplus \ad_f(\fu).\]

For grading \eqref{eq adh wieght decomp}, the following theorem relates the open dense $\rL$-orbit $\Oo_{-2}\subset\fg_{-2}$ with the $\fsl_2$-triple. It is a combination of results of Kostant and Malcev, see \cite[\S 10]{knappbeyondintro} or \cite[\S 3.4]{CollingwoodMcGovern} for details. 

\begin{theorem} \label{thm Kostant-Malcev open orbit of sl2 triple}
Let $\langle f,h,e\rangle\subset\fg$ be an $\fsl_2$-triple, and consider the $\Z$-grading $\fg=\bigoplus_{j\in\Z}\fg_j$ from \eqref{eq adh wieght decomp}. Let $\fp$ and $\fl$ be the associated parabolic and Levi subalgebras with associated Lie subgroups $\rL<\rP<\rG$. Then the nilpotent elements $e$ and $f$ are contained in the unique open dense $\rL$-orbits $\Oo_{2}\subset\fg_2$ and $\Oo_{-2}\subset\fg_{-2}$, respectively. Moreover, the $\rL$-stabilizer of $e$ is the Lie group $\rC<\rG$ which centralizes the $\fsl_2$-subalgebra. 
\end{theorem}

\subsection{Even Jacobson-Morozov parabolics}
\label{sec: even JM parabolics}
A parabolic subalgebra which arises from an $\fsl_2$-subalgebra is called a \emph{Jacobson-Morozov parabolic} (abbreviated JM-parabolic). Not all parabolic subalgebras are JM-parabolics. For example, the stabilizer of a line in $\C^n$ is not a JM-parabolic when $n>2.$ Moreover, different conjugacy classes of $\fsl_2$-triples can define the same JM-parabolic.  In this paper, the following notion of an \emph{even} $\fsl_2$-triple is essential.

\begin{definition} \label{def: even sl2}
    An $\fsl_2$-triple $\langle f,h,e\rangle\subset\fg$ is called \emph{even} if $\fg_1=0$ in the decomposition \eqref{eq adh wieght decomp}. 
\end{definition}

\begin{remark}
    \label{rem: odd dim sl2 reps for even sl2} 
    An $\fsl_2$-triple $\langle f,h,e\rangle\subset\fg$ is even if and only if $\fg_{2j+1}=0$ for all $j$ in decomposition \eqref{eq adh wieght decomp}. Equivalently, the $\fsl_2\C$-module decomposition from \eqref{eq sl2 module decomp} consists only of odd dimensional irreducible $\fsl_2$-representations. Another equivalent definition of $\fs=\langle f,h,e\rangle$ being even is that the associated subgroup $\rS<\rG_\Ad$ of the adjoint group is isomorphic to $\rPSL_2\C.$
\end{remark}



\begin{definition}\label{def: even JM parabolic}
A parabolic subgroup $\rP<\rG$ will be called an \emph{even JM-parabolic} if it arises from an even $\fsl_2$-triple. 
\end{definition}

\begin{remark}\label{rem: comparing gradings for even JM-parabolics}
Since the JM-parabolic $\fp$ associated to an $\fsl_2$-triple $\langle f,h,e\rangle$ comes with a choice of Levi subalgebra, we have two $\Z$-gradings. Namely, the $\ad_h$-weight space decomposition \eqref{eq adh wieght decomp} and the associated graded of the canonical filtration \eqref{eq associated graded Z grading of parabolic}. For even JM-parabolics, the $j^{th}$-graded piece of \eqref{eq associated graded Z grading of parabolic} is the $2j^{th}$-graded piece of \eqref{eq adh wieght decomp}. 
\end{remark}

The following proposition is an easy consequence of \cite[\S 3.8]{CollingwoodMcGovern}.
\begin{proposition}
There is a one-to-one correspondence between conjugacy classes of even $\fsl_2$-triples and conjugacy classes of even JM-parabolics.
\end{proposition}

Fix an even $\fsl_2$-triple $\langle f,h,e\rangle\subset\fg$. Let $1=m_1<m_2<\cdots<m_N$ be such that the $\fsl_2$-representations $W_{2m_j}$ from \eqref{eq sl2 module decomp} are nonzero. We have
\begin{equation}
    \label{eq even sl2 decomp data} \fg=\fc\oplus\bigoplus\limits_{j=1}^NW_{2m_j}.
\end{equation}

The Slodowy data from \eqref{eq highest weight space decomp} is 
\begin{equation}
    \label{eq: slodowy data}V=\fc\oplus \bigoplus_{j=1}^NV_{2m_j}.
\end{equation}

\begin{example}\label{ex: sl2 data for even JM in SL4}
    The first three parabolic subgroups of $\rSL_4\C$ described in Example \ref{ex: SL4 Lie data} are even JM-parabolics. In fact, up to conjugacy these are the only even JM-parabolic subgroups of $\rSL_4\C$. Below, we use the same numbering as the items in Example \ref{ex: SL4 Lie data}.

    \begin{enumerate}
        \item An associated $\fsl_2$-triple is 
        \[\fs=\langle f,h,e\rangle=\left\langle\left(\begin{smallmatrix}
            0&0&&\\1&0&0&\\&1&0&0\\&&1&0
        \end{smallmatrix}\right),\left(\begin{smallmatrix}
            3&0&&\\0&1&0&\\&0&-1&0\\&&0&-3
        \end{smallmatrix}\right),\left(\begin{smallmatrix}
            0&3&&\\0&0&4&\\&0&0&3\\&&0&0
        \end{smallmatrix}\right)\right\rangle.
        \]
        The centralizer of $\fs$ is the center, $\rC=\langle \sqrt{-1}\Id\rangle.$
        The associated $\fs$-module decomposition \eqref{eq sl2 module decomp} is
        \[\fsl_4\C=W_2\oplus W_4\oplus W_6,\]
        with multiplicities $(n_2,n_4,n_6)=(1,1,1).$ The Slodowy data is
        \[V_2\oplus V_4\oplus V_6=\left\langle\left(\begin{smallmatrix}
            0&3&0&0\\0&0&4&0\\0&0&0&3\\0&0&0&0
        \end{smallmatrix}\right) \right\rangle\oplus \left\langle\left(\begin{smallmatrix}
            0&0&1&0\\0&0&0&1\\0&0&0&0\\0&0&0&0
        \end{smallmatrix}\right)\right\rangle\oplus \left\langle \left(\begin{smallmatrix}
            0&0&0&1\\0&0&0&0\\0&0&0&0\\0&0&0&0
        \end{smallmatrix}\right)\right\rangle,\]
        and $\hat V_2=0$. 

        \item An associated $\fsl_2$-triple is given by $2\times 2$-block matrices
        \[\fs=\langle f,h,e\rangle=\left\langle\left(\begin{smallmatrix}
            0&0\\\Id&0
        \end{smallmatrix}\right),\left(\begin{smallmatrix}
            \Id&0\\0&-\Id
        \end{smallmatrix}\right),\left(\begin{smallmatrix}
            0&\Id\\0&0
        \end{smallmatrix}\right)\right\rangle,
        \]
        The centralizer of $\fs$ is given by 
        \[\rC=\left\{\left(\begin{smallmatrix}
            A&0\\0&A
        \end{smallmatrix}\right)~|~\det(A)^2=1\right\}\cong\rSL^\pm_2\C.\]
        The associated $\fs$-module decomposition \eqref{eq sl2 module decomp} is 
        \[\fsl_4\C=W_0\oplus W_2,\]
        with multiplicities $(n_0,n_2)=(3,4).$ The Slodowy data is 
        \[V_0\oplus V_2=\left(\begin{smallmatrix}
            X&0\\0&X
        \end{smallmatrix}\right)\oplus 
            \left(\begin{smallmatrix}
            0&Y\\0&0
        \end{smallmatrix}\right),\]
        where $\Tr(X)=0$, and the subspace $\hat V_2$ is given by $\Tr(Y)=0.$ 

     \item An associated $\fsl_2$-triple is given by 
        \[\fs=\langle f,h,e\rangle=\left\langle\left(\begin{smallmatrix}
            0&0&0&0\\1&0&0&0\\0&0&0&0\\0&1&0&0
        \end{smallmatrix}\right), \left(\begin{smallmatrix}
            2&0&0&0\\0&0&0&0\\0&0&0&0\\0&0&0&-2
        \end{smallmatrix}\right),\left(\begin{smallmatrix}
            0&2&0&0\\0&0&0&2\\0&0&0&0\\0&0&0&0
        \end{smallmatrix}\right) \right\rangle.\]
        The centralizer of $\fs$ is
        \[\rC=\left\{\left(\begin{smallmatrix}
            \alpha&&&\\&\alpha&&\\&&\alpha^{-3}&\\&&&\alpha
        \end{smallmatrix}\right)~|~\alpha\in\C^*\right\}\cong\rGL_1\C.\]
        The associated $\fs$-module decomposition \eqref{eq sl2 module decomp} is 
        \[\fsl_4\C=W_0\oplus W_2\oplus W_4,\]
        with multiplicities $(n_0,n_2,n_4)=(1,3,1).$ The Slodowy data is
        \[V_0\oplus V_2\oplus V_4=\left(\begin{smallmatrix}
            a&&&\\&a&&\\&&-3a&\\&&&a
        \end{smallmatrix}\right)\oplus \left(\begin{smallmatrix}
            0&x&y&0\\&&&x\\&&&z\\&&&0
        \end{smallmatrix}\right)\oplus \left(\begin{smallmatrix}
            0&0&0&w\\&&&0\\&&&0\\&&&0
        \end{smallmatrix}\right),\]
 and the subspace $\hat V_2$ is defined by setting $x=0.$ 
        \end{enumerate}
\end{example}

The next example sketches the construction and states some properties of principal $\fsl_2$-triples, for more information see \cite{ptds} or \cite[\S 4]{CollingwoodMcGovern}. 

\begin{example}
\label{ex: G/B principal sl2 data}  
The Borel subgroup $\rB<\rG$ is an even JM-parabolic. An associated even $\fsl_2$-triple is called a \emph{principal $\fsl_2$-subalgebra}. 
Fix a Cartan subalgebra $\fg_0\subset\fg$ and a set of simple roots $\{\alpha_1,\cdots,\alpha_{\rk(\fg)}\}.$ For each $\alpha_j$, choose $e_j\in\fg_{\alpha_j}$ and $f_j\in\fg_{-\alpha_j}$ so that $\langle f_j,[e_j,f_j],e_j\rangle$ is an $\fsl_2$-triple. Choose $h\in\fg_0$ so that $\alpha_{j}(h)=2$ for each $j$ and write $h=\sum_ja_j[e_j,f_j].$ Then 
\[\fs=\langle f,h,e\rangle=\left\langle \sum_j f_j~,~ h~,~ \sum_ja_je_j\right\rangle\]
is a principal $\fsl_2$-triple. In particular, note that the projection of $f$ onto each $\fg_{-\alpha_j}$ is nonzero, so $f\in\Oo_{-1}\subset \fg_{-1}.$

In this case, the centralizer $\rC<\rG$ of $\fs$ is the center of the group $\rG,$ thus $\fc=W_0=\{0\}.$ The associated $\fs$-module decomposition is $\fg=\bigoplus_j W_{2m_j}$ where $\{m_j\}$ are the exponents of $\fg.$ The multiplicities satisfy $n_{m_j}=1$ for all $j$, with the exception that $n_{4n-2}=2$ for $\fg=\fso_{2n}(\C)$.
\end{example}

Finally, we recall a generalization of a theorem of Kostant proven by Lynch in \cite[Theorem 1.2]{LynchThesis}. This theorem is the fundamental Lie theoretic ingredient which will allow us to understand $(\rG, \rP)$-opers for $\rP<\rG$ an even JM-parabolic in terms of simpler data, namely the Slodowy category.

\begin{theorem}\label{thm: Kostant generalization}
Let $\rP<\rG$ be an even JM-parabolic subgroup determined by an $\fsl_2$-triple $\langle f,h,e\rangle.$  Let 
$\rU^{m}< \rU^{m-1}<\cdots< \rU^{1}< \rP$ be the filtration of $\rP$ by unipotent subgroups and $\fg^{m}\subset\cdots\subset \fg^{0}=\fp$ be the corresponding canonical filtration of $\fp$.  
Then, for $V=\ker(\ad_e:\fg\to\fg),$ the map
\[\xymatrix@R=0em{\rU^i\times\{f+V\}\ar[r]&f+V+\fg^{i-1}\\(u,f+v)\ar@{|->}[r]&\Ad(u)(f+v)}\]
is an isomorphism of affine algebraic varieties for every $1\leq i\leq m$.
\end{theorem}

\begin{remark}
    In \cite[Lemma 2.1]{QuantizingSlodowy}, a more general version of Theorem \ref{thm: Kostant generalization} is proven for any (not necessarily even) $\fsl_2$-triple together with a choice of isotropic subspace of $\fg_1$ from \eqref{eq adh wieght decomp}. It would be interesting to consider how such data could be used to generalize our results to arbitrary parabolics.
\end{remark}

\section{Geometric preliminaries}

In this section, we rapidly review the theory of connections on principal bundles, and fix some notation which we shall use throughout the paper. 

\subsection{Bundles and connections}\label{sec: gauge theory}
Throughout this paper, all principal bundles are taken to be \emph{right} principal bundles.  If $\rG$ is a Lie group and $E_{\rG}$ is a principal $\rG$-bundle, then we denote the $\rG$-invariant vertical vector field on $E_{\rG}$ induced by $Y\in \fg$ by $Y^\sharp.$  

If $V$ is a complex manifold equipped with a \emph{left} holomorphic $\rG$-action, we will suppress the action and denote the associated holomorphic fiber bundle by $E_{\rG}[V],$ remembering that the structure group of this bundle is $\rG.$  The usual example arises from a Lie group acting on some vector subspace/quotient of a Lie algebra: for example, $E_{\rG}[\fg]$ is the holomorphic vector bundle determined by the adjoint action of $\rG$ on $\fg.$  

For the rest of this section, $E_{\rG}$ will be a holomorphic principal $\rG$-bundle over a \emph{Riemann surface} $X$ and $\Kk$ will denote the holomorphic cotangent bundle of $X.$  
\begin{definition}
A holomorphic connection $\omega$ on $E_\rG$ is a $\fg$-valued holomorphic $1$-form $\omega: TE_{\rG}\to \fg$ such that
\begin{enumerate}
\item $R_{g}^*\omega=\Ad(g^{-1})\circ \omega$ for all $g\in \rG.$ 
\item $\omega(Y^{\sharp})=Y$ for all $Y\in \fg.$  
\end{enumerate}
\end{definition}

The curvature $F(\omega)=d\omega+\frac{1}{2}[\omega, \omega]$ of a holomorphic connection $\omega$ defines an $E_{G}[\fg]$-valued holomorphic $(2,0)$-form on $X,$ and therefore always vanishes. Thus, all holomorphic connections on $E_\rG$ are flat (i.e. integrable), and we refer to the pair $(E_\rG,\omega)$ as a holomorphic flat $\rG$-bundle.

\begin{definition}
	The category $\Ff_{X}(\rG)$ of holomorphic flat $\rG$-bundles on $X$ has objects $(E_\rG,\omega)$, where 
	\begin{itemize}
		\item $E_\rG$ is a holomorphic $\rG$-bundle on $X,$
		\item $\omega$ is a holomorphic connection on $E_\rG.$
	\end{itemize}
	Morphisms $\Phi:(E_\rG,\omega)\to(F_\rG,\eta)$ are given by isomorphisms $\Phi: E_{\rG}\to F_{\rG}$ covering the identity map on $X,$ such that $\Phi^{*}\eta=\omega.$
\end{definition}



Let $\{U_{\alpha}\}$ be a trivializing open cover for a holomorphic principal $G$-bundle $E_{\rG}$ defined by local sections $s_{\alpha}: U_{\alpha}\to E_{\rG}$ with transition functions $g_{\alpha\beta}: U_{\alpha\beta}\to\rG$ where $U_{\alpha\beta}:=U_{\alpha}\cap U_{\beta}.$  
Given a holomorphic connection $\omega$ on $E_\rG$, define the local connection forms $\omega_{\alpha}:=s_{\alpha}^{*}\omega\in \Omega^{1}(U_{\alpha}, \fg).$  These local connection forms satisfy the relation
\begin{equation}
    \label{eq: transitions of local connections}\omega_{\alpha}=\Ad(g_{\alpha\beta}^{-1})\circ \omega_{\beta} +g_{\alpha\beta}^{*}\theta_{\rG},
\end{equation}
where $\theta_{\rG}$ is the left-invariant Mauer-Cartan form on $\rG.$  Moreover, given a set of trivializing local sections $\{s_{\alpha}\},$ any collection of locally defined holomorphic $1$-forms satisfying \eqref{eq: transitions of local connections} defines a holomorphic connection on $E_{\rG}.$

Finally, let $E_{\rG},F_\rG$ be  holomorphic principal $\rG$-bundles and $\eta$ be a holomorphic connection on $F_\rG.$ 
Let $s:U\to E_{\rG}$ and $t: U\to F_{\rG}$ be local sections over an open set $U\subset X.$  Consider a morphism $\Phi: E_{\rG}\to F_{\rG}$ and the local connection forms $\eta_{U}:=t^{*}\eta$ and $\eta_{U}^{\Phi}:=s^{*}(\Phi^{*}\eta).$
Then, with respect to the trivializations induced by $s$ and $t$, $\Phi$ is determined by a holomorphic map $\phi: U\to \rG$ and 
\begin{equation}
    \label{eq: local expresion for pullback of conn}\eta_{U}^{\Phi}=\Ad(\phi^{-1})\circ \eta_{U}+ \phi^{*}\theta_{\rG}.
\end{equation}

\subsection{Extensions and reductions of structure group}
Given a holomorphic principal $\rG_1$-bundle $E_{\rG_1}$ and a homomorphism $b:\rG_1\to\rG_2$ between complex Lie groups, there is an associated holomorphic principal $\rG_2$-bundle 
\[E_{\rG_1}[\rG_2]=(E_{\rG_1}\times\rG_2)/\sim,\]
where $(p,g_2)\sim(p\cdot g_1,b(g_{1}^{-1})\cdot g_2)$ for $g_1\in\rG_1$ and the action of $\rG_2$ on $E_{\rG_1}[\rG_2]$ is given by right multiplication on the second factor. 
Transition functions of $E_{\rG_1}[\rG_2]$ are defined by post-composing transition functions of $E_{\rG_1}$ with the homomorphism $b:\rG_1\to\rG_2.$
This process is usually called extending the structure group. 

Let $\omega$ be a holomorphic connection on $E_{\rG_1}$ and $\theta_{\rG_2}$ be the left invariant Maurer-Cartan form of $\rG_2$. Then the $\rG_2$-equivariant map 
\[\omega_b:\xymatrix@=0em{T(E_{\rG_1})\times T\rG_2\ar[r]&\fg_2\\(v,w)\ar@{|->}[r]&db(\omega(v))+\theta_{\rG_2}(w)}\] 
descends to define a connection on the extended bundle $E_{\rG_1}[\rG_2].$ We will denote the connection on $E_{\rG_1}[\rG_2]$ induced by $\omega$ by the same symbol $\omega.$

If $E_{\rG_1}$, $E_{\rG_2}$ are holomorphic $\rG_1$, $\rG_2$ bundles respectively, then their fiber product defines a holomorphic principal $(\rG_1\times\rG_2)$-bundle $E_{\rG_1\times\rG_2}.$
 Moreover, holomorphic connections $\omega,\eta$ on $E_{\rG_1},E_{\rG_2}$ respectively, induce a holomorphic connection $\omega\oplus\eta$ on $E_{\rG_1\times\rG_2}.$ 
When $\rG_1,\rG_2<\rG$ are two subgroups which commute with each other, then the multiplication map $m:\rG_1\times\rG_2\to\rG$ is a group homomorphism. In this case, two holomorphic bundles with holomorphic connection $(E_{\rG_1},\omega),(E_{\rG_2},\eta)$ induce a holomorphic $\rG$-bundle with connection which we denote by
\begin{equation}
    \label{eq: star notation}
    (E_{\rG_1}\star E_{\rG_2}[\rG],\omega\star\eta):=(E_{\rG_1\times\rG_2}[\rG],\omega\oplus\eta)~.
\end{equation}
Note that the transition functions of $E_{\rG_1}\star E_{\rG_2}[\rG]$
 are just the product of the transition functions of $E_{\rG_1}$ and $E_{\rG_2}$ in $\rG$ and that locally $\omega\star\eta$ is just a sum of the connection forms. 

Let $V$ be a complex vector space with a holomorphic $\rG$-action. A holomorphic $V$-valued $1$-form $\beta$ on $E_{G}$ is called \emph{horizontal} if $\beta(Y^\sharp)=0$ for every $Y\in \fg$, and \emph{equivariant} if $R_{g}^{*}\beta=g^{-1}\cdot \beta.$ 
There is a canonical isomorphism between equivariant, horizontal $V$-valued holomorphic $1$-forms and holomorphic sections the vector bundle $\Kk\otimes E_{\rG}[V]$ over $X.$

If $\rH<\rG$ is a complex Lie subgroup, then a holomorphic reduction of structure group of a holomorphic principal $\rG$-bundle $E_\rG$ to $\rH$ is a holomorphic sub-bundle $E_\rH\subset E_\rG:$ note that the inclusion homomorphism $H<G$ induces an isomorphism $E_{\rH}[\rG]\cong E_{\rG}.$
Given a holomorphic connection $\omega$ on $E_\rG$, the second fundamental form of $\omega$ relative to a reduction of structure $E_\rH\subset E_\rG$ measures the failure of $\omega$ to be induced by a connection on $E_\rH.$

\begin{definition}\label{def: second fundamental form of reduction}
Suppose $E_{\rH}\subset E_{\rG}$ is a holomorphic reduction of structure to a closed, complex subgroup $\rH<\rG.$  The second fundamental form $\Psi \in H^{0}(X, \Kk\otimes E_{\rH}[\fg/\fh])$ is the holomorphic section determined by the $\rH$-equivariant, horizontal $1$-form
\[\xymatrix{TE_{\rH}\ar[r]& TE_{\rG}\ar[r]^{\omega} &\fg\ar[r]& \fg/\fh}.\]
\end{definition}

Finally, we define a notion of relative position for a holomorphic connection relative to a reduction of structure.

\begin{definition}\label{def: relative pos of reduction}
Let $(E_{\rG},\omega)$ be a holomorphic flat $\rG$-bundle and $E_{\rH}\subset E_{\rG}$ be a holomorphic reduction to a complex subgroup $\rH<\rG$ with second fundamental form $\Psi.$  Let $\Oo\subset \fg/\fh$ be a $\C^*\times \rH$-invariant sub-manifold.
Then the position of $\omega$ relative to $E_{\rH}$ is equal to $\Oo$, denoted $\pos_{E_{\rH}}(\omega)=\Oo$, if and only if 
$\Psi(v)\in E_{\rH}(\Oo)$ for all non-zero vectors $v\in TX.$  
\end{definition}

Note, the $\C^{*}$-invariance implies that there is a well defined holomorphic fiber bundle $\Kk\otimes E_{H}[\Oo],$ and the second fundamental form $\Psi$ defines a holomorphic section of this fiber bundle.

\subsection{$\lambda$-connections and Higgs bundles}
In this subsection we generalize the notion of a holomorphic connection to depend on a parameter $\lambda\in\C.$ These objects were introduced by Deligne and have been studied by Simpson \cite{SimpsonHodgeFiltration}, and many others. Let $\rG$ be a complex semisimple Lie group and $E_\rG$ be a holomorphic principal $\rG$-bundle on a Riemann surface $X.$

\begin{definition}
Given $\lambda\in \C$, a holomorphic $\lambda$-connection on $E_{\rG}$ is a holomorphic $1$-form $\omega: TE_{\rG}\to \fg$ satisfying
\begin{enumerate}
\item $R_{g}^{*}\omega=\Ad(g^{-1})\circ \omega$ for all $g\in G.$
\item  $\omega(Y^{\sharp})=\lambda Y$ for all $Y\in \fg.$
\end{enumerate}
\end{definition}
\begin{remark}
 A holomorphic $1$-connection is just an ordinary holomorphic connection, and if $\omega$ is a $\lambda$-connection for $\lambda\neq 0,$ then
$\frac{1}{\lambda}\omega$ is an ordinary connection.  In contrast, a $0$-connection $\omega$ vanishes on vertical vectors, thus $\omega$ is equivalent to a holomorphic section $\widehat{\omega}\in H^{0}(X, \Kk\otimes E_{\rG}[\fg]),$ such a section is usually called a \emph{Higgs field}.
 \end{remark} 

\begin{remark}\label{rem:local trasformations of lambda conn}
    Recall that local connection forms transform as in \eqref{eq: transitions of local connections} and \eqref{eq: local expresion for pullback of conn}. Local $\lambda$-connections transform the same way with the Maurer-Cartan form $\theta_\rG$ in \eqref{eq: transitions of local connections} and \eqref{eq: local expresion for pullback of conn} replaced by $\lambda\theta_\rG.$
\end{remark}

\begin{definition}
    The category of holomorphic $(\rG,\lambda)$-connections $\Qq\Ff_{X}(\rG)$ on a Riemann surface $X$ has objects $(\lambda,E_\rG,\omega)$, where 
    \begin{itemize}
        \item $\lambda\in\C$,
        \item $E_\rG\to X$ is a holomorphic principal $\rG$-bundle,
        \item $\omega$ is a holomorphic $\lambda$-connection on $E_\rG$.
    \end{itemize}
    A morphism between two objects $(\lambda_1,E_\rG,\omega)$ and $(\lambda_2,F_\rG,\eta)$ is an isomorphism $\Phi:E_\rG\to F_\rG$ covering the identity on $X$ such that $\Phi^*\eta=\omega.$ In particular, there are no such morphisms when $\lambda_1\neq\lambda_2.$
\end{definition}

Note that there is a projection $\pi:\Qq\Ff_{X}(\rG)\to\C$ defined by $\pi(\lambda,E_\rG,\omega)=\lambda.$ For a fixed $\lambda\in \C,$ the fiber $\pi^{-1}(\lambda)=\Ff^\lambda(\rG)$ defines a full subcategory $\Ff_{X}^{\lambda}(\rG)\subset \Qq\Ff_{X}(\rG).$  There is also a $\C^*$-action on $\Qq\Ff_{X}(\rG)$:
\begin{equation}
    \label{eq: C*action on lambda conn}
    \xi\cdot(\lambda,E_\rG,\omega)=(\xi\lambda,E_\rG,\xi\omega),
\end{equation}
and $\pi$ is equivariant with respect to the standard action of $\C^*$ on $\C.$ 

When $\lambda=0,$ the category of $0$-connections $\Ff_{X}^0(\rG)$ is the well-known \emph{category of $\rG$-Higgs bundles} on $X.$ In this case, a holomorphic zero connection $(0,E_\rG,\omega)$ is equivalent to the data $(E_\rG,\hat\omega)$ where $\hat\omega\in \rH^0(X, \Kk\otimes \Ee_\rG[\fg]).$ 

Given a homogeneous basis $\{p_1,\cdots,p_{\rk(\fg)}\}$ of the $\rG$-invariant polynomials on $\fg$ with $\deg(p_j)=m_j+1$, we can define the \emph{Hitchin fibration}:
\begin{equation}
    \label{eq: Hitchin fibration} H:\xymatrix@R=0em{\Ff_{X}^0(\rG)\ar[r]&\bigoplus\limits_{j=1}^{\rk(\fg)}\rH^0(X,\Kk^{m_j+1})\\(E_\rG,\hat\omega)\ar@{|->}[r]&(p_1(\hat\omega),\cdots,p_{\rk(\fg)}(\hat\omega))}.
\end{equation}
The vector space $\bigoplus_{j=1}^{\rk(\fg)}\rH^0(X, \Kk^{m_j+1})$ is called the \emph{Hitchin base}. In \cite{liegroupsteichmuller}, Hitchin used the special features of principal $\fsl_2$-triples to construct a section of this map.  Since one of the main innovations in this paper is to put the work of Hitchin and Beilinson-Drinfeld in a more general context, we sketch the construction of the \emph{Hitchin section} for adjoint groups $\rG_\Ad$ here.

Let $\langle f,h,e\rangle\subset\fg$ be the principal $\fsl_2$-triple from Example \ref{ex: G/B principal sl2 data}. Recall that the Lie algebra decomposes into $\ad_h$-eigenspaces $\fg=\bigoplus_{j\in\Z}\fg_{2j}.$ Consider the holomorphic Lie algebra bundle 
\[\Ee_\fg:=\bigoplus\limits_{j\in\Z}\Kk^{j}\otimes\fg_{2j}.\]
The holomorphic vector bundle $\Ee_\fg$ is canonically associated to a principal $\rG_{\textnormal{Ad}}$-bundle, so that $E_{\rG_\Ad}[\fg]$ is canonically isomorphic to $\Ee_\fg.$ 


Recall that $V=\ker(\ad_e)=\bigoplus_{j=1}^{\rk(\fg)}V_{2m_j}$ is a sum of highest weight spaces where $V_{2m_j}=V\cap\fg_{2mj}.$ 
Choose a non-zero vector $e_{2m_j}$ in each highest weight space $V_{2m_j}$. 
Note that $f\in \fg_{-2}$ defines a holomorphic section of $\Kk\otimes(\Kk^{-1}\otimes\fg_{-2})$ and $q_{m_j+1}\otimes e_{2m_j}$ defines a holomorphic section of $\Kk\otimes (\Kk^{m_j}\otimes V_{2m_j})\subset\Kk\otimes (\Kk^{m_j}\otimes \fg_{2m_j})$ for each $q_{m_j+1}\in \rH^0(X,\Kk^{m_j+1})$. 
Consider the map $s_H:\bigoplus_{j=1}^{\rk(\fg)}\rH^0(X,\Kk^{m_j+1})\to\Ff_X^0(\rG_\Ad)$ defined
by
\begin{equation}
    \label{eq: Hitchin section}s_H(q_{m_1+1},\cdots,q_{m_{\rk(\fg)}+1})=\left(E_{\rG_\Ad},f+\sum_{j=1}^{\rk(\fg)}q_{m_j+1}\otimes e_{2m_j}\right).
\end{equation}
In \cite{KostantPolyBasis}, Kostant proved that there is a homogeneous basis $(p_{1},\cdots,p_{\rk(\fg)})$ of the $\rG$-invariant polynomials on $\fg$ with $\deg(p_j)=m_j+1$ such that 
\[p_{j}\left(f+\sum_{j=1}^{\rk(\fg)}a_{2m_j}e_{2m_j}\right)=a_{2m_j}.\]
Defining the Hitchin fibration with respect to this choice of basis for the invariant polynomials makes $s_H$ a section of the map $H$ in equation \eqref{eq: Hitchin fibration}.

\section{The category of (G,P)-opers}
In this section we define the main new object of the paper: $(\rG,\rP)$-opers. Let $X$ be a connected Riemann surface, $\rG$ be a connected complex semisimple Lie group, and $\rP<\rG$ be a parabolic subgroup.
Let $\Oo\subset \fg^{-1}/\fp$ be the unique open $\rP$-orbit on $\fg^{-1}/\fp$ from \S \ref{sec: parabolics}.  Note that $\Oo$ is also $\C^*$-invariant.
\begin{definition}\label{Def oper def}
A $(\rG, \rP)$-oper on $X$ is a triple $(E_{\rG}, E_{\rP}, \omega)$ such that
\begin{enumerate}
\item  $(E_{\rG}, \omega)$ is a holomorphic flat $G$-bundle over $X,$
\item $E_{\rP}\subset E_{\rG}$ is a holomorphic reduction of structure to the parabolic subgroup $\rP<\rG,$
\item the position of $\omega$ relative to $E_{\rP}$ satisfies $\pos_{E_{\rP}}(\omega)=\Oo.$
\end{enumerate}
\end{definition}
\begin{remark}
Note that there is an isomorphism $E_{\rP}[\fg^{-1}/\fp]\cong E_{\rL}[\fg_{-1}]$ where $\rL\cong \rP/\rU$ is the Levi factor of $\rP.$  
  When $\rP=\rB$, using Example \ref{ex: G/B principal sl2 data}, the third condition of Definition \ref{Def oper def} is that the projection of $\omega$ to every negative simple root space is nowhere vanishing.  This recovers Beilinson-Drinfeld's definition of a $\rG$-oper in \cite{BeilinsonDrinfeldOPERS}. 
\end{remark}
A morphism between two $(\rG, \rP)$-opers $(E_{\rG}, E_{\rP}, \omega), (F_{\rG}, F_{\rP}, \eta)$ is an isomorphism $\Phi:E_\rG\to F_\rG$ of holomorphic principal bundles such that $\Phi|_{E_\rP}:E_\rP\to F_\rP$ is an isomorphism and $\Phi^*\eta=\omega$. 
The category of $(\rG,\rP)$-opers over $X$ is denoted $\Op_{X}(\rG,\rP).$  
\begin{remark}
    The definition of a $(\rG,\rP)$-oper extends immediately to the context of $\lambda$-connections. We denote the category of $(\rG,\rP, \lambda)$-opers over $X$ as $\lambda$ ranges over $\C$ by $\Qq\Op_{X}(\rG,\rP).$ Namely, the objects of $\Qq\Op_{X}(\rG,\rP)$ consists of tuples $(\lambda, E_{\rG}, E_{\rP}, \omega)$, where $\lambda\in\C$ and $(E_\rG,E_\rP,\omega)$ is a $\lambda$-oper.
\end{remark}

Using Example \ref{ex: SL4 Lie data}, we can rephrase the above definition for $\rG=\rSL_4\C$ in terms of rank $4$-vector bundles.
\begin{example}
\label{ex: oper definition for even JM in SL4}
Given a holomorphic $\rSL_4\C$-bundle $\Ee$ with connection $\omega$, let $(\Vv,\nabla, \Omega)$ be the associated holomorphic vector bundle with connection $\nabla$ and parallel volume $\Omega$. 
\begin{enumerate}
    \item Since $\rB$ is the stabilizer of full flag in $\C^4,$ a holomorphic $\rB$-reduction $\Ee_\rB\subset\Ee$ defines a holomorphic filtration 
    \[0=\Vv_0\subset\Vv_1\subset\Vv_2\subset\Vv_3\subset\Vv_4=\Vv,\]
    where $\rk(\Vv_i)=i.$
    The third condition of Definition \ref{Def oper def} is that, for each $i,$ $\nabla(\Vv_i)\subset\Vv_{i+1}\otimes\Kk$ and $\nabla$ induces a bundle isomorphism
    \[\Vv_i/\Vv_{i-1}\xrightarrow{\cong}\Vv_{i+1}/\Vv_i\otimes \Kk.\]
    \item For $\rP$ the stabilizer of a subspace $\C^2\subset\C^4$, a holomorphic $\rP$-reduction $\Ee_\rP\subset\Ee$ defines a holomorphic filtration
    \[0=\Vv_0\subset\Vv_2\subset\Vv_4=\Vv,\]
    where $\rk(\Vv_i)=i.$ The third condition of Definition \ref{Def oper def} is that $\nabla$ induces a bundle isomorphism 
    \[\Vv_2\xrightarrow{\cong}\Vv/\Vv_2\otimes \Kk.\]
    \item For $\rP$ the stabilizer of a partial flag $\C^1\subset\C^3\subset\C^4$, a holomorphic $\rP$-reduction $\Ee_\rP\subset\Ee$ defines a holomorphic filtration
    \[0=\Vv_0\subset\Vv_1\subset\Vv_3\subset\Vv_4=\Vv,\]
    where $\rk(\Vv_i)=i$. The third condition of Definition \ref{Def oper def} is that $\nabla(\Vv_1)\subset \Vv_3 \otimes \Kk$ and the composition
    \[\Vv_1\xrightarrow{\nabla} \Vv_{3}\otimes \Kk \xrightarrow{\nabla\otimes\textnormal{id}} \Vv\otimes \Kk^2 \to \Vv/\Vv_{3} \otimes \Kk^{2}\]
     is a bundle isomorphism.
    \item For $\rP$ the stabilizer of a subspace $\C^1\subset\C^4$, a holomorphic $\rP$-reduction $\Ee_\rP\subset\Ee$ defines a holomorphic line sub-bundle $\Vv_1\subset\Vv.$ The third condition of Definition \ref{Def oper def} is that the map $\Vv_1\to\Vv/\Vv_1\otimes \Kk$ induced by $\nabla$ is an injective bundle map.
\end{enumerate}
\end{example}
\begin{remark}
    Each of the above examples generalizes naturally to higher rank holomorphic vector bundles. For the first example, we recover the classical vector bundle definition of an oper. Namely, an $(\rSL_n\C,\rB)$-oper consists of a rank $n$ holomorphic vector bundle $\Vv$ with holomorphic volume form $\Omega,$ equipped with a holomorphic filtration $\Vv_1\subset\cdots\subset\Vv_{n-1}\subset\Vv$ with $\rk(V_i)=i,$ and a holomorphic connection $\nabla$ preserving $\Omega$ and satisfying
   $\nabla(\Vv_i)\subset\Vv_{i+1}\otimes \Kk$ and $\nabla:\Vv_{i}/\Vv_{i-i}\to\Vv_{i+1}/\Vv_i\otimes\Kk$ an isomorphism for all $i$.
\end{remark}

If $\rG=\rS$ is isomorphic to $\rP\rSL_2\C$ or $\rSL_2\C,$ the Borel $\rB_\rS<\rS$ is the only proper parabolic subgroup. In terms of vector bundles, an $\rSL_2\C$-oper is a $4$-tuple $(\Vv,\Vv_1,\nabla, \Omega)$, where $\rk(\Vv)=2$ and $\Omega$ is a holomorphic volume form on $\Vv.$  Furthermore, $\Vv_{1}\subset \Vv$ is a holomorphic line sub-bundle and $\nabla$ is a holomorphic connection preserving $\Omega$ such that the induced map $\nabla:\Vv_1\to\Vv/\Vv_1\otimes \Kk$ is an isomorphism. 
When $X$ is a compact Riemann surface of genus $g\geq 2,$ this implies $\Vv_1\cong\Kk^\frac{1}{2}$ is one of the $2^{2g}$-square roots of $\Kk$ and $\Vv/\Vv_1\cong \Kk^{-\frac{1}{2}}.$ 
The ambient holomorphic vector bundle $\Vv$ is isomorphic to the unique non-split extension
\[\xymatrix{0\ar[r]&\Kk^\frac{1}{2}\ar[r]&\Vv\ar[r]&\Kk^{-\frac{1}{2}}\ar[r]& 0}.\]

Choosing a smooth identification of $\Vv$ with $\Kk^\frac{1}{2}\oplus \Kk^{-\frac{1}{2}},$ an explicit example of an oper is given by the following Dolbeault operator and holomorphic connection: 
\begin{equation}
    \label{eq psl2 oper in smooth splitting}\bar\partial_\Vv=\begin{pmatrix}\bar\partial_{1/2}&h\\0&\bar\partial_{-1/2}
\end{pmatrix}\ \ \ \ \ \ \ \ \text{and} \ \ \ \ \ \ \ \ \nabla=\begin{pmatrix}
    \nabla^h_{1/2}&q\\1&\nabla^h_{-1/2}
\end{pmatrix}.
\end{equation}
Here, $h \in \mathcal{A}^{0,1}(X, \Kk)$ is the uniformizing hyperbolic metric on $X$, $\bar\partial_{1/2}$ and $\bar\partial_{-1/2}$ are Dolbeault operators on $K^\frac{1}{2}$ and $K^{-\frac{1}{2}}$ respectively, $\nabla^h_{1/2}$ and $\nabla^h_{-1/2}$ are the $(1,0)$-parts of the Chern connections of the induced metrics on $K^\frac{1}{2}$ and $K^{-\frac{1}{2}}$, respectively, and $q\in \rH^0(X, \Kk^2).$ In fact, every $\rSL_2\C$-oper is isomorphic to \eqref{eq psl2 oper in smooth splitting}. This will be generalized in Theorem \ref{thm: parameterization picking PSL2 oper} for $(\rG,\rP)$-opers when $\rP$ is an even JM-parabolic. 

We can also use \eqref{eq psl2 oper in smooth splitting} to construct sections of $\Qq\Op_{X}(\rSL_2\C)\to \C.$ One such section is given by 
\begin{equation} 
    \label{eq sections of lambda}\tau(\lambda)=(\lambda,\bar\partial_\Ff,\nabla)=\left(\lambda,\begin{pmatrix}\bar\partial_{1/2}&\lambda h\\0&\bar\partial_{-1/2}
\end{pmatrix}, \begin{pmatrix}
    \lambda\nabla^h_{1/2}&\lambda^2 q\\1&\lambda \nabla^h_{-1/2}
\end{pmatrix}\right)~,
\end{equation}

\begin{remark}
    In the above discussion, we have a choice of a square root of $\Kk.$ Each such choice defines the same $\rPSL_2\C$-oper.
\end{remark}

\section{The global Slodowy slice for even JM-parabolics}
In this section, we define the $\fs$-Slodowy category $\Bb_{\fs, X}(\rG)$ associated to any even $\fsl_2$-triple $\fs=\langle f,h,e\rangle\subset\fg$.  Throughout the rest of the paper, we fix once and for all a connected, compact Riemann surface $X$ of genus $g\geq2$.  All holomorphic objects are taken over $X.$  

We first prove that the category $\Bb_{\fs, X}(\rG)$ is equivalent to the category $\Op_{X}(\rG,\rP)$ of $(\rG,\rP)$-opers, where $\rP$ is the even JM-parabolic associated to $\fs$. This equivalence can be interpreted as a ``global slice theorem''. When $\fs$ is a principal subalgebra and $\rG$ is an adjoint group, $\Bb_{\fs, X}(\rG)$  is the Hitchin base and we recover Beilinson-Drinfeld's parameterization of $(\rG,\rB)$-opers.

Let $\rS<\rG$ be the connected subgroup with Lie algebra $\fs.$ The functor realizing the equivalence depends on the choice of an $\rS$-oper. We then show that there is a full subcategory $\widehat\Bb_{\fs, X}(\rG)$ of $\Bb_{\fs, X}(\rG)$ and a functor $\Op_{X}(\rS,\rB_\rS)\times\widehat\Bb_{\fs, X}(\rG)\to \Op_{X}(\rG,\rP)$ which is an equivalence when $\rS\cong\rPSL_2\C,$ and essentially surjective and full when $\rS\cong\rSL_2\C.$ Finally we extend these results to the category of $(\lambda,\rG,\rP)$-opers. For $\lambda=0$ and $\rP=\rB$, we recover the Hitchin section.

\subsection{The $\fs$-Slodowy category}
Let $\rG$ be a connected, complex semi-simple Lie group. Fix an even $\fsl_2$-triple $\fs=\langle f,h,e\rangle\subset\fg$, let $\fb_\fs=\langle h,e\rangle\subset\fg$ and let $\rB_\rS<\rS<\rG$ be the associated connected subgroups of $\rG$. Recall from \eqref{eq even sl2 decomp data} that we have an $\fs$-module decomposition of $\fg$ 
\[\fg=\fc\oplus\bigoplus_{j=1}^NW_{2m_j}.\]
Recall also that the Slodowy data of $\fs$ given by 
\[\xymatrix{V=\fc\oplus\bigoplus\limits_{j=1}^NV_{2m_j}}.\]
Moreover, from Proposition \ref{prop: s+b_0 decomp of V and V_2}, we have $V_2=\langle e\rangle\oplus \hat V_2.$

The Slodowy data of $\fs$ is the Lie theoretic mechanism which defines the generalization of the Hitchin base. 
\begin{definition}\label{def s Slodowy category}
Let $\rG$ be a connected complex semi-simple Lie group and $\fs=\langle f,h,e\rangle\subset\fg$ be an even $\fsl_2$-triple with centralizer $\rC<\rG$. The \emph{$\fs$-Slodowy} category $\Bb_{\fs, X}(\rG)$ has objects $(E_\rC,\psi_0,\psi_{m_1},\cdots,\psi_{m_N}),$ where 
\begin{itemize}
    \item $E_\rC$ is a holomorphic $\rC$-bundle.
    \item $\psi_0$ is a holomorphic $\rC$-connection on $E_\rC$.
    \item $\psi_{m_j}\in  \rH^0(X,E_\rC[V_{2m_j}]\otimes \Kk^{m_j+1}).$
\end{itemize}
The morphisms are given by isomorphisms of holomorphic flat $\rC$-bundles which identify holomorphic sections of $E_{\rC}[V_{2m_j}]\otimes \Kk^{m_j+1}$ for every $j.$
\end{definition}

\begin{remark}
    For a principal $\fsl_2$-subalgebra $\fs\subset\fg$, $\rC<\rG$ is the center of $\rG$ and $\psi_{m_j}\in \rH^0(X, \Kk^{m_j+1})$. In this case, $\Bb_{\fs, X}(\rG)$ is the product of the Hitchin base with the finite set of holomorphic $\rC$-bundles on $X$. In particular, for an adjoint group $\rG$, $\Bb_{\fs, X}(\rG)$ is the Hitchin base. 
\end{remark}


The following lemma allows us to view the \emph{coefficients} $\psi_{m_{j}}$ in the Slodowy category as holomorphic $1$-forms valued in some auxiliary bundle, which are the kind of objects that can be added to connections.

\begin{lemma}
Let $\Theta=(F_{\rS},F_{\rB_\rS},\eta)$ be an $(\rS,\rB_\rS)$-oper.
For every holomorphic $\rC$-bundle $E_\rC,$ the second fundamental form of $\eta$ induces an isomorphism
\begin{equation}
    \label{eq iso levi to slodowy}E_{\rC}\star F_{\rB_S}[V_{2m_j}]\otimes \Kk\cong E_\rC[V_{2m_j}]\otimes \Kk^{m_j+1},
\end{equation}
where we recall the $\star$-notation from \eqref{eq: star notation}.
\end{lemma}
\begin{proof}
The oper condition implies that the second fundamental form of $\eta$ relative to $F_{\rB_{\rS}}$ induces an isomorphism
$\Kk^{-1}\cong F_{\rB_{\rS}}[\fs/{\fb_{\fs}}].$  This implies that $\Kk\cong F_{\rB_{\rS}}[\langle e \rangle].$  Since $\rB_{\rS}$ acts on $V_{2m_{j}}$ via a multiplicative character of exponent $2m_{j},$ this implies that $\Kk^{m_{j}}\otimes V_{2m_{j}}\cong F_{\rB_{\rS}}[V_{2m_{j}}].$  
\end{proof}

We now define the \emph{Slodowy} functor, which should be viewed as a generalized affine deformation of the Hitchin section.  
\begin{definition}
    \label{def Slodowy functor}
Let $\rG$ be a connected complex semisimple Lie group, $\fs=\langle f,h,e\rangle\subset\fg$ an even $\fsl_2$-triple, and $\rS<\rG$ the associated connected subgroup. For each $(\rS,\rB_\rS)$-oper $\Theta=(F_{\rS},F_{\rB_\rS},\eta)$, the \emph{$\Theta$-Slodowy functor} 
\begin{equation}
    \label{eq psl2 oper functor}F_{\Theta}: \Bb_{\fs, X}(\rG)\to\Op_{X}(\rG,\rP) 
\end{equation}
is defined by
\[F_\Theta(\Xi)=(E_{\rC}\star F_{\rS}[\rG],E_\rC\star F_{\rB_\rS}[\rP],\eta\star\psi_0+\psi_{m_1}+\cdots+\psi_{m_N}),\]
where $\Xi=(E_\rC,\psi_0,\psi_{m_1},\cdots,\psi_{m_N})$ is an object in $\Bb_{\fs, X}(\rG)$ and we have identified each $\psi_{m_j}$ with a holomorphic section of $F_{\rB_\rS}\star E_\rC[V_{2m_j}]\otimes \Kk$ using the isomorphism \eqref{eq iso levi to slodowy}. If $\Phi:E_\rC\to E_\rC'$ is a morphism in $\Bb_{\fs, X}(\rG)$, then $F_\Theta(\Phi)=\Phi\star\Id_{F_\rS}:E_\rC\star F_\rS(\rG)\to E_\rC'\star F_\rS(\rG).$  
\end{definition}

Note that the Slodowy functor \emph{does} produce a $(\rG,\rP)$-oper, since the resulting holomorphic connection leaves the $\fg_{-1}$-space unchanged, and therefore the second fundamental form of $\eta\star\psi_0+\psi_{m_1}+\cdots+\psi_{m_N}$ relative to $E_\rC\star F_{\rB_\rS}[\rP]$ is identified with the second fundamental form of $\eta.$  

\begin{remark}
    The Slodowy functor can be extended to a functor 
   \begin{equation}
       \label{eq slodowy functor product}F:\xymatrix@R=0em{\Op_{X}(\rS,\rB_\rS)\times \Bb_{\fs, X}(\rG)\ar[r]&\Op_{X}(\rG,\rP) \\(\Theta,\Xi)\ar@{|->}[r]&F_\Theta(\Xi)}.
   \end{equation}
    Given morphisms $\Phi_1:\Theta\to\Theta'$ and $\Phi_2:\Xi\to\Xi'$, we have $F(\Phi_1,\Phi_2)=\Phi_1\star \Phi_2:F(\Theta,\Xi)\to F(\Theta',\Xi').$
\end{remark}

Recall that a functor defines an equivalence of categories if it is essentially surjective and fully faithful, i.e. every isomorphism class is in the image and the functor induces a bijection on morphisms.  The central theorem of this paper is the following \emph{parameterization} of $(\rG,\rP)$-opers via the Slodowy category.

\begin{theorem}\label{thm: equivalence of categories}
Let $\rG$ be a connected complex semi-simple Lie group, $\fs=\langle f,h,e\rangle\subset \fg$ be an even $\fsl_2$-triple, and $\rP<\rG$ be the associated even JM-parabolic.
For every $(\rS,\rB_\rS)$-oper $\Theta$, the $\Theta$-Slodowy functor
\[F_{\Theta}: \Bb_{\fs, X}(\rG)\to\Op_{X}(\rG,\rP) \]
is an equivalence of categories.  
\end{theorem}

\begin{remark}
  When $\fs$ is a principal subalgebra and $\rP=\rB$ is the Borel subgroup, Theorem \ref{thm: equivalence of categories} is due to Beilinson-Drinfeld \cite{BeilinsonDrinfeldOPERS}. Note that, in this case, the Slodowy functor is defined similarly to the Hitchin section \eqref{eq: Hitchin section}, see \eqref{eq: Slodowy functor recovering Hitchin section} for the explicit relation.  
\end{remark}


To remove the choice of $(\rS,\rB_\rS)$-oper from Theorem \ref{thm: equivalence of categories}, we introduce a full sub-category $\widehat\Bb_{\fs, X}(\rG)$ of $\Bb_{\fs, X}(\rG)$. Since the space $V_2$ decomposes $\rC$-invariantly as $\langle e\rangle\oplus \hat V_2$, for any holomorphic $\rC$-bundle $E_\rC$ we have 
\[E_\rC[V_2]\otimes \Kk^2\cong \Kk^2\oplus (E_\rC[\hat V_2]\otimes \Kk^2).\]
Thus, for an object $(E_\rC,\psi_0,\psi_{m_1},\cdots,\psi_{m_N})$ in $\Bb_{\fs, X}(\rG)$, $\psi_{m_1}$ decomposes as 
\begin{equation}
    \label{eq psi1 decomp}\psi_{m_1}=q\oplus\hat\psi_{m_1},
\end{equation}
where $q\in \rH^0(X,\Kk^2)$ and $\hat\psi_{m_1}\in \rH^0(X,E_C[ \hat V_2]\otimes \Kk^2).$ 
\begin{definition}
    Let $\fs=\langle f,h,e\rangle\subset\fg$ be an even $\fsl_2$-triple. Define the \emph{traceless quadratic $\fs$-Slodowy category} to be the full subcategory $\widehat\Bb_{\fs, X}(\rG)$ of $\Bb_{\fs, X}(\rG)$ whose objects are $(E_\rC,\psi_0,\psi_{m_1},\cdots,\psi_{m_N})$ with $\psi_{m_1}=\hat\psi_{m_1}.$
\end{definition}

The following theorem mirrors Theorem \ref{thm: equivalence of categories}, but reinstates the symmetry broken by the choice of a fixed $(\rS, \rB_{\rS})$-oper.

\begin{theorem}\label{thm: no quadratic equivalence}
    Let $\rG$ be a connected complex semi-simple Lie group, $\fs=\langle f,h,e\rangle\subset\fg$ an even $\fsl_2$-triple, $\rS<\rG$ the connected subgroup with Lie algebra $\fs$, and $\rP<\rG$ the associated even JM-parabolic. Then, the functor 
    \[F:\xymatrix@R=0em{\Op_{X}(\rS,\rB_\rS)\times \widehat\Bb_{\fs, X}(\rG)\ar[r]&\Op_{X}(\rG,\rP) \\(\Theta,\Xi)\ar@{|->}[r]&F_\Theta(\Xi)}\]
   is an equivalence of categories when $\rS\cong\rPSL_2\C$, and essentially surjective and full when $\rS\cong\rSL_2\C$.
\end{theorem}
\begin{remark}
    Recall that when $\rG$ is an adjoint group, the subgroup $\rS<\rG$ is always isomorphic to $\rPSL_2\C$. Thus, for adjoint groups, the above functor is an equivalence of categories.
\end{remark} 
\subsection{Proofs of Theorems \ref{thm: equivalence of categories} and \ref{thm: no quadratic equivalence} }
In this section we will prove Theorems \ref{thm: equivalence of categories} and \ref{thm: no quadratic equivalence}. These proofs will be relatively straightforward consequences of Lemma \ref{lem local statement for functor}.  The reader is encouraged to skip to Theorem \ref{thm: parameterization picking PSL2 oper} to understand the usage of this technical statement. 

Fix an even $\fsl_2$-triple $\fs=\langle f,h,e\rangle\subset\fg$ and retain the notation from the previous sections. In particular, $\rP<\rG$ is the even JM-parabolic associated $\fs$ and $\rU<\rP$ is the unipotent radical which has a filtration $\rU^{m_N}<\cdots <\rU^1=\rU.$ Throughout this section we will use the $\Z$-grading \eqref{eq associated graded Z grading of parabolic}, i.e., $\fg_j\subset\fg$ is the $j^{th}$ eigenspace of $\ad_{\frac{h}{2}}.$  Recall the open orbits $\Oo_{-1}\subset \fg_{-1}$ and $\Oo_{\fs}=\C^{*}\cdot f,$ and the vector space $V=\textnormal{ker}(\textnormal{ad}_{e})\subset \fg.$  

Before proving Lemma \ref{lem local statement for functor} we prove two auxillary lemmas.

\begin{lemma}\label{lemma local no indices}
Let $U$ be a $1$-connected domain and consider holomorphic $1$-forms $\omega:TU\to\fg_{-1}\oplus\fp$ and $\eta:TU\to\fs$ such that $\omega(v)\in\Oo_{-1}+\fp$ and $\eta(v)\in\Oo_\fs+\fb_\fs$ for all nonzero $v\in TU.$ Then there exists a holomorphic map $\Psi:U\to\rP$ such that 
\begin{equation}\label{eq: local difference lambda=1}
	\psi:=\Ad(\Psi^{-1})\circ\omega+\Psi^*\theta_\rP-\eta: TU\to V,
\end{equation}
where $\theta_{\rP}$ is the left invariant Mauer-Cartan form on $\rP.$ 
\end{lemma}
\begin{proof}
Since $\rP$ acts transitively on $\Oo_{-1}+\fp$ and $U$ is $1$-connected, there is a holomorphic map $\Psi_{0}: 
U\to \rP$ such that 
\[\Ad(\Psi_{0}^{-1})\circ \omega - \eta:TU\to  \fp.
\]
Define
\[
\omega_{0}:=\Ad(\Psi_{0}^{-1})\circ\omega + \Psi_{0}^{*}\theta_{\rL}:TU\to \Oo_{\fs} + \fp.
\]
By Theorem \ref{thm: Kostant generalization}, there exists a unique holomorphic $\Psi_{1}: U\to \rU^{1}$ such that
\[
\Ad(\Psi_{1}^{-1})\circ \omega_{0}: U\to \Oo_{\fs} + V
\]
Then, define
\[
\omega_{1}:=\Ad(\Psi_{1}^{-1})\circ\omega_{0} + \Psi_{1}^{*}\theta_{\rU^{1}}:TU\to \Oo_{\fs} +V+ \fu^{1} .
\]

Using Theorem \ref{thm: Kostant generalization} inductively for each $i\geq2$, we obtain a unique $\Psi_{i}: U\to \rU^{i}$ such that 
$\Ad(\Psi_{i}^{-1})\circ \omega_{i-1}: TU\to \Oo_{\fs}+ V,$
 and the resulting
\[
\omega_{i}:=\Ad(\Psi_{i}^{-1})\circ\omega_{i-1} + \Psi_{i}^{*}\theta_{\rU^{i}}:TU\to \Oo_{\fs} +V+ \fu^{i}.
\]
When $i=m_N,$ since $\fg^{m_N}\subset V,$ we obtain $\omega_{m_N}: TU\to \Oo_{\fs}+V.$  

Then $\Psi=\Psi_{0}\cdots \Psi_{m_N+1}:U\to \rP$ is the desired map since 
\begin{align*} 
\Ad(\Psi^{-1})\circ \omega + \Psi^{*}\theta_{\rP}&=\Ad({\Psi_{m_N+1}^{-1}})\circ \omega_{m_N}-\Ad(\Psi_{m_N+1}^{-1}\cdot\dots\cdot \Psi_{1}^{-1})\circ \Psi_{0}^{*}\theta_{L} \\
&-\sum_{i=1}^{m_N} \Ad(\Psi_{m_N+1}^{-1}\cdot\dots\cdot \Psi_{i+1}^{-1})\circ \Psi_{i}^{*}\theta_{\rU^{i}} + \Psi^{*}\theta_{\rP} \\
&=\Ad(\Psi_{m_N+1}^{-1})\circ \omega_{m_N} + \Psi_{m_N+1}^{*}\theta_{\rU_{\rS}}.
\end{align*}
\end{proof}
We also need the following Lie theoretic lemma. 
\begin{lemma}\label{lemma Ad_U(f) not in V}
	Let $u\in\rU$  
	\begin{enumerate}
		\item If $u=\exp(x_{j}+\cdots+x_{m_N})$, where $x_k\in\fg_{k},$ then  $\Ad(u)(f)=f-\ad_f(x_j)+b$ for $b\in\fu^j$.
		\item  Suppose $u$ is in the group generated by $\rU_\rS$ and $\rU^j$ for $j\geq 2.$ If the projection of $\Ad(u)(f)$ onto $\ad_f(\fu^j)\cap\fg_{j-1}$ is zero, then $u$ is in the group generated by $\rU_\rS$ and $\rU^{j+1}.$
	\end{enumerate}
\end{lemma}
\begin{proof}
	The first part follows from a direct computation. We have
	\[\Ad(u)(f)=\sum_{i=0}^\infty\frac{\ad_x^i(f)}{i!}=f-\ad_f(x_j)+b,\]
	where $b\in\fu^j$ is given by $b=-\sum_{k=j+1}^{m_N}\ad_f(x_k)+\sum_{i=2}^\infty\frac{\ad_x^i(f)}{i!}$.
	

 For the second part, suppose $u$ is in the group generated by $\rU_\rS$ and $\rU^j$ for $j\geq 2$. Then $u=\exp(\lambda e+x_j+\cdots+x_{m_N})$ for $\lambda\in\C$ and $x_k\in\fg_k.$ A computation shows that
\[\Ad(u)(f)=\sum_{i=0}^\infty\frac{\ad_x^i(f)}{i!}=f+\ad_{\lambda e}(f)+\frac{\ad_{\lambda e}^2(f)}{2}-\ad_f(x_j)+b,\]
where $b\in\fu^j$. The projection onto $\ad_f(\fu^j)\cap\fg_{j-1}$ is $-\ad_f(x_j),$ and is zero if and only if $x_j=0$. Hence, this projection is zero if and only if $u$ is in the group generated by $\rU_\rS$ and $\rU^{j+1}.$
\end{proof}
	\begin{lemma}\label{lem local statement for functor}
	Let $U_{\alpha},U_{\beta}\subset \C$ be $1$-connected domains such that $U_{\alpha\beta}:=U_{\alpha}\cap U_{\beta}\neq \emptyset$ is $1$-connected. For $i=\alpha,\beta$, let $\omega_i$, $\eta_i$, $\Psi_i$ and $\psi_i$ be as in Lemma \ref{lemma local no indices}. Suppose there exists holomorphic maps $p_{\ab}:U_{\alpha\beta}\rightarrow \rP$ and $b_{\ab}:U_{\alpha\beta}\rightarrow \rB_{\rS}$ such that
\[\xymatrix{\omega_{\alpha}=\Ad(p_{\ab}^{-1})\circ \omega_{\beta} + p_{\ab}^*\theta_{\rP}&\text{and}&
\eta_{\alpha}=\Ad(b_{\ab}^{-1})\circ \eta_{\beta} + b_{\ab}^*\theta_{\rB_{\rS}}}.\]
 If $q_{\ab}:=\Psi_{\beta}^{-1}\cdot p_\ab\cdot \Psi_{\alpha}:U_{\alpha\beta}\rightarrow \rP$, then
 \begin{equation}\label{preserves connection}
\Ad(q_{\ab}^{-1})(\eta_{\beta} + \psi_{\beta}) + q_{\ab}^{*}\theta_{P}=\eta_{\alpha}+\psi_{\alpha}.
 \end{equation}
Furthermore, the holomorphic map $q_{\ab}\cdot b_\ab^{-1}: U_\ab\to\rP$ is valued in $\rC$.
	\end{lemma}
	\begin{proof}
By assumption we have $\psi_i+\eta_i=\Ad(\Psi_i^{-1})\circ \omega_i+\Psi_i^*\theta_\rP$ for $i=\alpha,\beta$. Equation \eqref{preserves connection} follows from the definition of $q_\ab$:
\begin{align*}
\Ad(q_\ab^{-1})\circ(\eta_{\beta} + \psi_{\beta}) &=\Ad(\Psi_{\alpha}^{-1}\cdot p_\ab^{-1} \cdot \Psi_{\beta})(\Ad(\Psi_{\beta}^{-1})\circ \omega_{\beta}+\Psi_{\beta}^{*}\theta_{\rP}) \\
&= \Ad(\Psi_{\alpha}^{-1})\circ \omega_{\alpha} + \Psi_{\alpha}^{*}\theta_{\rP} - q_\ab^{*}\theta_{P} \\
&=\eta_{\alpha} + \psi_{\alpha} - q_\ab^{*}\theta_{\rP}.
\end{align*}

Rewriting \eqref{preserves connection} leads to
    \begin{equation}
        \label{eq alpha beta difference} 
        \psi_\alpha-\Ad(q_\ab^{-1})\circ\psi_\beta+\eta_\alpha-\Ad(q_\ab^{-1})\circ\eta_\beta=q_\ab^*\theta_\rP.
    \end{equation}Since $\rP\cong\rL\ltimes\rU$, there are unique holomorphic maps $ l_\ab:U_\ab\to \rL$ and $u_\ab:U_\ab\to\rU$ such that $q_\ab=u_\ab\cdot  l_\ab$. 

Equation \eqref{eq alpha beta difference} is a direct sum of graded pieces. By hypothesis, the $(-1)$-graded piece is the first non-zero term. Since the $(-1)$-graded pieces of $\psi_\alpha,$ $\Ad(q_\ab^{-1})\circ\psi_\beta$ and $q_\ab^*\theta_\rP$ vanish, we have
    \begin{equation}\label{eq -1 piece}
    	(\eta_\alpha)_{-1}-(\Ad(q_\ab^{-1})\circ\eta_\beta)_{-1}~=~(\eta_\alpha)_{-1}-\Ad( l_\ab^{-1})\circ(\eta_\beta)_{-1}~=~0.
    \end{equation}
Thus, $ l_\ab=\widehat l_{\rS,\ab}\cdot c_\ab$ for holomorphic maps $c_\ab:U_\ab\to\rC$ and $\widehat l_{\rS,\ab}:U_\ab\to\rL_\rS$, where $\rL_\rS<\rB_\rS$ is the Levi associated to $\langle h\rangle\subset\fb_\fs.$

For the zeroth graded piece we have 
    \begin{equation}\label{eq 0 graded piece}(\eta_\alpha)_{0}-(\Ad(q_\ab^{-1})\circ\eta_\beta)_{0}+(\psi_\alpha)_0-(\Ad(q_\ab^{-1})\circ\psi_\beta)_0=c_\ab^*\theta_\rP+\widehat l_{\rS,\ab}^*\theta_\rP.
    \end{equation}
The term $(\psi_\alpha)_0-(\Ad(q_\ab^{-1})\circ\psi_\beta)_0$ takes values in $\fc.$ Since $(q_\ab^*\theta_\rP)_0$ is valued $\langle h\rangle\oplus\fc,$ the term $(\eta_\alpha)_{0}-(\Ad(q_\ab^{-1})\circ\eta_\beta)_{0}$ is valued in $\langle h\rangle.$ Thus, 
    \begin{equation}\label{eq C transitions}
	(\psi_\alpha)_0-(\Ad(q_\ab^{-1})\circ\psi_\beta)_0=(\psi_\alpha)_0-\Ad(c_\ab^{-1})\circ(\psi_\beta)_0=c_\ab^*\theta_\rP.
\end{equation}
Since $\rC$ and $\rL_\rS$ act trivially on $\langle h\rangle$, the other term is given by
\begin{equation}
	\label{eq 0 term}(\eta_\alpha)_0-(\eta_\beta)_0-\Ad(\widehat l_{\rS,\ab}^{-1})\circ(\Ad(u_\ab)\circ\eta_\beta)_0=\widehat l_{\rS,\ab}^*\theta_\rP.
\end{equation}
Since $(\eta_\beta)_{-1}$ is valued in $\langle f\rangle$, Lemma \ref{lemma Ad_U(f) not in V} implies $u_\ab$ is valued in the group generated by $\rU_\rS$.

Now, the first graded piece is valued in $\langle e\rangle$ and given by 
\[(\eta_\alpha)_{1}-(\Ad(q_\ab^{-1})\circ\eta_\beta)_{1}+(\psi_\alpha)_1-(\Ad(q_\ab^{-1})\circ\psi_\beta)_1=(q_\ab^*\theta_\rP)_1.\]
Since $\psi_\beta$ is valued in $V$ and $q_\ab$ is valued in the group generated by $\rC,~\rB_\rS$ and $\rU^2,$ the term $(\psi_\alpha)_1-(\Ad(q_\ab^{-1})\circ\psi_\beta)_1$ is valued in $V$. Hence the projection of $(\Ad(q_\ab^{-1})\circ\eta_\beta)_{1}$ onto $\ad_f(\fu^2)$ is zero. In particular, this implies that the projection of $\Ad(u_\ab^{-1})\circ (\eta_\beta)_{-1}$ onto $\ad_f(\fu^2)$ is zero. Hence, $u_\ab$ takes values in the group generated by $\rU_\rS$ and $\rU^3$ by Lemma \ref{lemma Ad_U(f) not in V}.

 
For $j\geq 2$, assume $u_\ab$ is valued in the group generated by $\rU_\rS$ and $\rU^{j+1}$. Then $0=(q_\ab^*\theta_P)_j$ and the term $(\psi_\alpha)_j-(\Ad(q_\ab^{-1})\circ\psi_\beta)_j$ is valued in $V.$ Thus, $(\Ad(q_\ab^{-1})\circ\eta_\beta)_{j}=0$. By Lemma \ref{lemma Ad_U(f) not in V}, the image of $q_\ab$ is contained in the subgroup of $\rP$ generated by $\rC,$ $\rB_\rS$ and $\rU^{j+2}$, and 
\begin{equation}
    \label{eq psi_j trans like tensor}(\psi_\alpha)_j-(\Ad(q_\ab^{-1})\circ\psi_\beta)_j=0.
\end{equation}
Since $\rU^{m_N+1}=0,$ $q_\ab$ is valued in the subgroup generated by $\rB_\rS$ and $\rC.$ 

Hence, we have shown that $q_{\ab}$ factors as $c_{\ab}\cdot \widehat b_\ab$, for holomorphic maps $c_\ab:U_\ab\to\rC$ and $\widehat b_\ab:U_\ab\to\rB_\rS$. We also have $b_\ab:U_\ab\to \rB_\rS$ such that $\eta_\alpha=\Ad(b_\ab^{-1})\circ\eta_\beta+b_\ab^*\theta_{\rB_\rS}.$ Moreover, there are unique holomorphic maps $l_{\rS,\ab},\widehat l_{\rS,\ab}: U_\ab\to\rL_\rS$ and $u_{\rS,\ab}, \widehat u_{\rS,\ab}:U_\ab\to\rU_\rS$ such that $b_\ab=l_{\rS,\ab}\cdot u_\ab$ and $\widehat b_\ab=\widehat l_{\rS,\ab}\cdot\widehat  u_\ab$, respectively. 

Using \eqref{eq -1 piece}  and the fact that $\rC$ acts trivially on $\Oo_\fs$, we have 
\[\Ad(\widehat l_{\rS,\ab})\circ (\eta_\beta)_{-1}-\Ad(l_{\rS,\ab})\circ (\eta_\beta)_{-1}=0.\]
Since $\rL_\rS$ acts transitively on $\Oo_\fs$ with stabilizer the center of $\rS,$ $l_{\rS,\ab}$ and $\widehat l_{\rS,\ab}$ differ by an element of the center of $\rS.$ In particular, $l_{\rS,\ab}^*\theta_{\rL_\rS}=\widehat l_{\rS,\ab}^*\theta_{\rL_\rS}.$ 

Now using \eqref{eq 0 term}, we have 
\[\Ad(\widehat l_{\rS,\ab}^{-1})(\Ad(\widehat u_{\rS,\ab}^{-1})\circ \eta_\beta)_{0}-\Ad(l_{\rS,\ab}^{-1})(\Ad(u_{\rS,\ab}^{-1})\circ \eta_\beta)_{0}=0.\]
Note that $\Ad(\widehat l_{\rS,\ab}^{-1})=\Ad(l_{\rS,\ab}^{-1})$ and, for $u\in\rU_\rS,$
\[(\Ad(u)\circ \eta_\beta)_{0}=(\eta_\beta)_0+(\Ad(u)\circ(\eta_\beta)_{-1})_0.\]
Thus, 
\[(\Ad(u_{\rS,\ab})\circ(\eta_\beta)_{-1})_0-(\Ad(\widehat u_{\rS,\ab})\circ(\eta_\beta)_{-1})_0=0,\]
and hence $u_{\rS,\ab}=\widehat u_{\rS,\ab}.$ 

From the above argument $\widehat b_{\ab}\cdot b_{\ab}^{-1}$ is valued in the center of $\rS$ and hence in $\rC.$ Thus $q_\ab=c_\ab\cdot \widehat b_\ab$ satisfies $q_\ab\cdot b_\ab^{-1}:U_\ab\to\rC.$ 
\end{proof}

The proofs of Theorems \ref{thm: equivalence of categories} and \ref{thm: no quadratic equivalence} follow almost immediately from the following parameterization theorem.
\begin{theorem}\label{thm: parameterization picking PSL2 oper}
Let $\rG$ be a connected complex semisimple Lie group, $\fs=\langle f,h,e\rangle\subset\fg$ be an even $\fsl_2$-triple, $\rS<\rG$ be the connected subgroup with Lie algebra $\fs,$ $\rC<\rG$ be the centralizer of $\fs$ and $\rP<\rG$ be the associated even JM-parabolic. 
Let $\Theta=(F_{\rS}, F_{\rB_{\rS}}, \eta)$ be an $(\rS,\rB_\rS)$-oper and $(E_{\rG}, E_{\rP}, \omega)$ a $(\rG,\rP)$-oper.  Then, there exists a flat $\rC$-bundle $(E_{\rC}, \psi_0)$ and an isomorphism
$\Psi: (E_{\rC}\star F_{\rB_{\rS}})[\rP]\rightarrow E_{\rP}$
such that 
\begin{equation}\label{global normal form}
\Psi^{*}\omega-\eta\star \psi_0=\sum_{j=1}^{N}\psi_{m_j}\in \bigoplus_{j=1}^{N}H^{0}(X, \Kk\otimes \left(E_{\rC}\star F_{\rB_{\rS}}[V_{2m_j}]\right)).
\end{equation}
Finally, if $(\widetilde{E_{\rC}}, \widetilde{\psi_{0}})$ and $\widetilde{\Psi}$ is another flat $\rC$-bundle and isomorphism, respectively, which satisfies \eqref{global normal form}, then the isomorphism $\widetilde{\Psi}^{-1}\Psi: (E_{\rC}\star F_{\rB_{\rS}})[\rP]\rightarrow (\widetilde{E_{\rC}}\star F_{\rB_{\rS}})[\rP]$ induces an isomorphism in the Slodowy category
between $(E_{\rC}, \psi_0, \sum_{j=1}^{N}\psi_{m_j})$ and $(\widetilde{E_{\rC}}, \widetilde{\psi_{0}}, \sum_{j=1}^{N} \widetilde{\psi_{m_j}}).$  

\end{theorem}
\begin{proof}
    Let $\{U_\alpha\}$ be a trivializing open cover of both $E_{\rP}$ and $F_{\rB_\rS}$ and let 
    \[\xymatrix{\{p_\ab:U_\ab\to \rP\}&\text{and}&\{b_\ab:U_\ab\to \rB_\rS\}}\]
    be the transition functions of $E_\rP$ and $F_{\rB_\rS}$ respectively. The restrictions of $\omega$ and $\eta$ to $U_\ab$ satisfy
    \[\xymatrix{\omega_\alpha=\Ad(p_\ab^{-1})\circ\omega_\beta+p_\ab^*\theta_\rP&\text{and}&\eta_\alpha=\Ad(b_\ab^{-1})\circ\eta_\beta+b_\ab^*\theta_{\rB_\rS}}.\]

    By Lemma \ref{lemma local no indices}, we obtain $\{\Psi_\alpha:U_\alpha\to \rP\}$ so that 
    \[\Ad(\Psi_\alpha^{-1})\circ\omega_\alpha+\Psi_\alpha^*\theta_\rP-\eta_\alpha:=\psi_\alpha:TU_\alpha\to V.\]
    Let $E_{\Qq}$ be the principal $\rP$-bundle defined by the transition functions $\Qq:=\{q_\ab=\Psi_\beta^{-1} p_\ab \Psi_\alpha:U_\ab\to \rP\}.$  By construction, the locally defined $\Psi_{\alpha}$ patch to a globally defined isomorphism $\Psi: E_{\mathcal{Q}}\rightarrow E_{\rP}.$ 
    Now, we are exactly in the setting of Lemma \ref{lem local statement for functor} and we deduce that $q_{\ab}=c_{\ab}\cdot b_{\ab}$ for some holomorphic map $c_{\ab}: U_{\ab}\rightarrow \rC.$

   Since $\rB_\rS$ and $\rC$ commute,  the functions $c_\ab: U_{\ab}\rightarrow \rC$ are the transition functions of a holomorphic $\rC$-bundle $E_\rC.$ 
 Therefore, $E_\Qq=E_\rC\star F_{\rB_\rS}[\rP]$ and $\Psi:  E_\rC\star F_{\rB_\rS}[\rP]\rightarrow E_{\rP}$ is the promised isomorphism.

Next, by Lemma \ref{lem local statement for functor} equation \eqref{preserves connection}, the locally defined holomorphic $1$-forms $\{\eta_{\alpha}+\psi_{\alpha}\}$ glue to define a holomorphic connection on
$E_\rC\star F_{\rB_\rS}[\rG]$ which, by construction, is equal to $\Psi^{\star}\omega.$  

By equation \eqref{eq 0 graded piece}, $\{(\psi_\alpha)_0\}$ defines a holomorphic connection on $E_\rC.$ Moreover, by \eqref{eq psi_j trans like tensor} each $\{(\psi_{\alpha})_{m_j}\}$ defines a  section $\psi_{m_j}\in \rH^0(X, E_\rC\star F_{\rB_\rS}[V_{2m_j}]\otimes\Kk).$
Thus,
\[\Psi^*\omega=\eta\star\psi_0+\psi_{m_1}+\cdots+\psi_{m_N}.\]

Finally, assume we are given a $\widetilde{\Psi}$ as in the statement of the Theorem which satisfies
\[\widetilde{\Psi}^*\omega=\eta\star\widetilde{\psi_0}+\widetilde{\psi_{m_1}}+\cdots+
\widetilde{\psi_{m_N}}.\]
Then by breaking into graded pieces we immediately obtain $(\widetilde{\Psi}^{-1}\circ \Psi)^{*}(\eta\star \widetilde{\psi}_{0})=\eta\star \psi_{0}$ and $(\widetilde{\Psi}^{-1}\circ \Psi)^{*}(\widetilde{\psi}_{i})=\psi_{i}$ for all $i>0.$  Therefore, this defines a morphism in the Slodowy category.
\end{proof}

We now prove Theorem \ref{thm: equivalence of categories} and Theorem \ref{thm: no quadratic equivalence}.
\begin{proof}[Proof of Theorem \ref{thm: equivalence of categories}]
By Theorem \ref{thm: parameterization picking PSL2 oper}, for each $(\rS,\rB_\rS)$-oper $\Theta,$ the Slodowy functor $F_{\Theta}$ is essentially surjective. Given a morphism $\Phi:(E_\rC,\psi_0)\to(E_\rC',\psi_0')$ in $\Bb_{\fs, X}(\rG)$, $F_\Theta(\Phi)$ is defined by $\Phi\star\Id_{\rB_\rS}:E_\rC\star F_{\rB_\rS}[\rG]\to E_\rC'\star F_{\rB_\rS}[\rG].$ This is clearly faithful. 

The fullness follows from the final statement of Theorem \ref{global normal form}.
\end{proof}

\begin{proof}[Proof of Theorem \ref{thm: no quadratic equivalence}]
    Fix an $(\rS,\rB_\rS)$-oper $\Theta=(F_\rS,F_{\rB_\rS},\eta)$ and let $(E_\rG,E_\rP,\omega)$  be a $(\rG,\rP)$-oper. By Theorem \ref{thm: parameterization picking PSL2 oper}, there is $\Xi=(E_\rC,\psi_0,\psi_1,\cdots,\psi_{m_N})$ in  $\Bb_{\fs, X}(\rG)$ such that 
    \[(E_\rG,E_\rP,\omega)\cong (F_\rS\star E_\rC[\rG],F_{\rB_\rS}\star E_\rC[\rP],\eta\star\psi_0+\psi_{m_1}+\cdots+\psi_{m_N}),\]
    where, using the isomorphism \eqref{eq iso levi to slodowy}, $\psi_{m_j}\in \rH^0(X,F_{\rB_\rS}\star E_{\rC}[V_{2m_j}] \otimes \Kk).$

    By \eqref{eq psi1 decomp}, $\psi_{m_1}$ decomposes as $\psi_{m_1}=q+\hat\psi_{m_1}$, where 
    \[q\in \rH^0(X,E_\rC\star F_{\rB_\rS}[\langle e\rangle]\otimes\Kk) \cong \rH^0(X, F_{\rB_\rS}[\langle e\rangle]\otimes\Kk)~,\]
    \[\hat\psi_{m_1}\in  \rH^0(X, E_\rC\star F_{\rB_\rS}[\hat V_2]\otimes \Kk).\]
    Let $\hat\Xi=(E_\rC,\psi_0,\hat\psi_{m_1},\psi_{m_2},\cdots,\psi_{m_N})$ be the associated object in $\widehat\Bb_{\fs, X}(\rG).$ 
    Note that $\Theta_q=(F_\rS,F_{\rB_\rS},\eta+q)$ is another $(\rS,\rB_\rS)$-oper and 
    \[F_{\Theta}(\Xi)=F_{\Theta_q}(\hat\Xi).\]
    Thus, the functor $F:\Op_{X}(\rS,\rB_\rS)\times \widehat\Bb_{\fs, X}(\rG)\to \Op_{X}(\rG,\rP)$ defined by $F(\Theta,\hat\Xi)=F_\Theta(\hat\Xi)$ is essentially surjective. 

For fullness, consider
\[\Phi:E_\rC\star F_{\rB_\rS}[\rP]\to E_\rC'\star F_{\rB_\rS}'[\rP]\]
such that $\Phi^*(\eta'\star\psi_0'+\hat\psi_1'+\cdots+\psi_{m_N}')=\eta\star\psi_0+\hat\psi_1+\cdots+\psi_{m_N}.$ Locally, $\Phi=\{\Phi_\alpha:U_\alpha\to\rP\}$, where each $\Phi_\alpha$ preserves $\fs+V$. By Lemma \ref{lemma Ad_U(f) not in V}, we have  $\Phi_\alpha:U_\alpha\to\rC\star\rB_\rS$, and $\Phi:E_\rC\star F_{\rB_\rS}\to E_\rC'\star F_{\rB_\rS}'$. Such a $\Phi$ can be written as a product $\Phi_\rC\cdot \Phi_{\rB_\rS}$ where $\Phi_\rC: E_\rC\to E_\rC'$ and $\Phi_{\rB_\rS}:F_{\rB_\rS}\to F_{\rB_\rS}'$. 

For faithfulness, note that the multiplication map $\rS\times\rC\to\rG$ is injective if and only if $\rS\cong\rPSL_2\C.$ Thus, the functor $F$ is faithful if and only if $\rS\cong\rPSL_2\C.$
\end{proof}


\subsection{Explicit models for $\rSL_4\C$-opers}\label{sec: explicit models}

Recall Examples \ref{ex: SL4 Lie data}, \ref{ex: sl2 data for even JM in SL4} and \ref{ex: oper definition for even JM in SL4}. For $\rSL_4\C,$ the $\fs$-Slodowy categories are defined as follows.
\begin{enumerate}
    \item For $\fs=\langle f,h,e\rangle \subset\fsl_4\C$ a principal $\fsl_2,$ we have $C=\langle \sqrt{-1}\Id\rangle.$ The objects of $\Bb_{\fs, X}(\rSL_4\C)$ consists of tuples 
    \[(\Ll,\psi_1,\psi_2,\psi_3),\]
    where $\Ll$ is a holomorphic line bundle on $X$ such that $\Ll^4=\Oo_X$ and $\psi_j\in \rH^0(X,\Kk^{j+1})$ for $j=1,2,3.$

    \item For $\fs=\langle f,h,e\rangle \subset\fsl_4\C$ the even $\fsl_2$ whose associated JM-parabolic is the stabilizer of subspace $\C^2\subset\C^4$, we have $\rC\cong\rSL_2^\pm\C$. Using Example \ref{ex: sl2 data for even JM in SL4}, the objects of $\Bb_{\fs, X}(\rSL_4\C)$ are tuples
    \[(\Ww,\nabla_\Ww,\psi_1),\]
    where $\Ww$ is a rank two holomorphic vector bundle with $\det(\Ww)^2\cong\Oo_X,$ $\nabla_\Ww$ is a holomorphic connection on $\Ww$ compatible with the isomorphism $\det(\Ww)^2\cong\Oo_X,$  and $\psi_1\in \rH^0(X,\End(\Ww)\otimes \Kk^2).$ The decomposition \eqref{eq psi1 decomp} is $\psi_1=q\otimes \Id_\Ww+\hat\psi_1,$
    where $q\in \rH^0(X,\Kk^2)$ and $\hat\psi_1\in \rH^0(X,\End(\Ww)\otimes\Kk^2)$ is a traceless.

    \item For $\fs=\langle f,h,e\rangle \subset\fsl_4\C$, the even $\fsl_2$ whose associated JM-parabolic is the stabilizes a partial flag $\C\subset\C^3\subset\C^4$, we have $\rC\cong\rGL_1\C$. Using Example \ref{ex: sl2 data for even JM in SL4}, the objects of $\Bb_{\fs, X}(\rSL_4\C)$ are tuples
    \[(\Ll,\nabla_\Ll,\psi_1,\psi_2),\]
    where $\Ll$ is a degree zero line bundle, $\nabla_\Ll$ is a holomorphic connection on $\Ll$, $\psi_{1}\in \rH^0(X,\Kk^2)\oplus \rH^0(X,\Ll^4\Kk^2)\oplus \rH^0(X,\Ll^{-4}\Kk^2)$ and $\psi_2\in \rH^0(X,\Kk^3).$ The decomposition \eqref{eq psi1 decomp} of $\psi_{1}$ is $\psi_{1}=q+\hat\psi_1,$ where $q\in \rH^0(X,\Kk^2)$ and $\hat\psi_1\in \rH^0(X,\Ll^4\Kk^2)\oplus \rH^0(X,\Ll^{-4}\Kk^2).$ 
\end{enumerate}

We now describe the Slodowy functor for each of these examples. To make the descriptions more concrete we will choose a smooth isomorphism between filtered objects and their associated graded. The holomorphic bundles with connections will then be expressed as Dolbeault operators and connections on the smooth bundle. 


Recall Examples \ref{ex: oper definition for even JM in SL4} and  Example \ref{ex: sl2 data for even JM in SL4}. Let $\Theta$ the $\rSL_2\C$-oper from \eqref{eq psl2 oper in smooth splitting}. The ordering of the below cases is the same as the previous examples. 
    \begin{enumerate}
    \item The $\Theta$-Slodowy functor $F_\Theta:\Bb_{\fs,X}(\rSL_4\C)\to\Op_{X}(\rSL_4\C,\rB)$ is  
    \[F_\Theta(\Ll,\psi_1,\psi_2,\psi_3)=(\bar\partial_\Vv,\nabla)
    \]
    where $(\bar\partial_\Vv,\nabla)$ are Dolbeault operators and holomorphic connections on the smooth bundle $\Ll \Kk^{\frac{3}{2}}\oplus \Ll \Kk^{\frac{1}{2}}\oplus \Ll \Kk^{-\frac{1}{2}}\oplus \Ll \Kk^{-\frac{3}{2}}$ given by
    \[\bar\partial_\Vv=\begin{pmatrix}
        \bar\partial_\Ll\otimes\bar\partial_{3/2}&3h&0&0\\0&\bar\partial_\Ll\otimes\bar\partial_{1/2}&4h&0\\0&0&\bar\partial_\Ll\otimes\bar\partial_{-1/2}&3h\\0&0&0&\bar\partial_\Ll\otimes\bar\partial_{-3/2}\end{pmatrix}\]
        \[\nabla=\begin{pmatrix}
        \nabla_\Ll\otimes\nabla^h_{3/2}&3q+3\psi_1&\psi_2&\psi_3\\1&\nabla_\Ll\otimes\nabla_{1/2}^h&4q+4\psi_1&\psi_2\\0&1&\nabla_\Ll\otimes\nabla^h_{-1/2}&3q+3\psi_1\\0&0&1&\nabla_\Ll\otimes\nabla^h_{-3/2}\end{pmatrix},\]
        where $\nabla_\Ll$ is the holomorphic connection which induces the trivial connection on $\Ll^4=\Oo_X.$

    \item  The $\Theta$-Slodowy functor $F_\Theta:\Bb_{\fs,X}(\rSL_4\C)\to\Op_{X}(\rSL_4\C,\rP)$ is 
    \[F_\Theta(\Ww,\nabla_\Ww,\psi_1)=(\bar\partial_\Vv,\nabla),
    \]
    where $(\bar\partial_\Vv,\nabla)$ are Dolbeault operators and holomorphic connections on the smooth bundle $(\Ww\otimes \Kk^\frac{1}{2})\oplus (\Ww\otimes \Kk^{-\frac{1}{2}})$ given by
    \[\bar\partial_\Vv=\begin{pmatrix}
        \bar\partial_{\Ww}\otimes \bar\partial_{1/2}&h\otimes \Id_\Ww\\0&\bar\partial_{\Ww}\otimes \bar\partial_{1/2}\end{pmatrix}\ \ \ \   \text{and} \ \ \ \ \nabla=\begin{pmatrix}
            \nabla_\Ww\otimes \nabla_{1/2}^h&q\otimes\Id_\Ww+\psi_1\\\Id_\Ww&\nabla_\Ww\otimes\partial_{-1/2}^h
        \end{pmatrix}.\]

    \item  The $\Theta$-Slodowy functor $F_\Theta:\Bb_{\fs,X}(\rSL_4\C)\to\Op_{X}(\rSL_4\C,\rP)$ is
    \[F_\Theta(\Ll,\nabla_\Ll,\psi_1,\psi_2)=(\bar\partial_\Vv,\nabla)
    \]
    where $(\bar\partial_\Vv,\nabla)$ are Dolbeault operators and holomorphic connections on the smooth bundle $\Ll K\oplus \Ll \oplus\Ll^{-3}\oplus \Ll \Kk^{-1}$ given by
    \[\bar\partial_\Vv=\begin{pmatrix}
        \bar\partial_\Ll\otimes\bar\partial_{1}&h&0&0\\0&\bar\partial_\Ll&0&h\\0&0&\bar\partial_{\Ll^{-3}}&0\\0&0&0&\bar\partial_\Ll\otimes\bar\partial_{-1}\end{pmatrix}\]
        \[\nabla=\begin{pmatrix}
        \nabla_\Ll\otimes\nabla^h_{1}&2q+2\alpha&\beta&\psi_3\\1&\nabla_\Ll&0&2q+2\alpha\\0&0&\nabla_{\Ll^{-3}}&\gamma\\0&1&0&\nabla_\Ll\otimes\nabla^h_{-1}\end{pmatrix},\]
        where $\psi_1=(\alpha,\beta,\gamma)\in \rH^0(X,\Kk^2)\oplus \rH^0(X,\Ll^4\Kk^2)\oplus \rH^0(X,\Ll^{-4}\Kk^2).$
\end{enumerate}

\subsection{The $(\fs,\lambda)$-Slodowy functor}
We now describe appropriate generalizations of Theorems \ref{thm: equivalence of categories}, \ref{thm: no quadratic equivalence} and \ref{thm: parameterization picking PSL2 oper} to the category $\Qq\Op_{X}(\rG,\rP) $ of $(\lambda,\rG,\rP)$-opers. Since the proofs and statements of the results are almost identical to the $\lambda=1$ case, we mostly omit the details.

First note that the definition of the $\fs$-Slodowy category $\Bb_{\fs, X}(\rG)$ generalizes immediately to a $(\lambda,\fs)$-Slodowy category $\Qq\Bb_{\fs, X}(\rG)$ whose objects consist of tuples $(\lambda,E_\rC,\psi_0,\psi_{m_1},\cdots,\psi_{m_N})$, where $(\lambda,E_\rC,\psi_0)$ is an object in $\Qq\Ff(\rC)$ and $\psi_{m_j}\in \rH^0(X,E_\rC[V_{2m_j}]\otimes \Kk^{m_j+1}).$ Similarly, we define the traceless quadratic part $\Qq\widehat\Bb_{\fs, X}(\rG)$ to be the locus where $\psi_{m_1}=\hat\psi_{m_1}.$

There are natural projections,
    \[\xymatrix{\Qq\Op_{X}(\rG,\rP) \to \C&\text{and}&\Qq\Bb_{\fs, X}(\rG)\to\C}.\]
 Denote the fiber category over $\lambda\in\C$ by $\Op_{X}^\lambda(\rG,\rP)$ and $\Bb_{\fs_,X}^\lambda(G)$, respectively. 

 The Slodowy functor from Definition \ref{def Slodowy functor} also immediately generalizes to $\lambda$-connections. For each $(\lambda,\rS,\rB_\rS)$-oper $\Theta_\lambda=(\lambda,F_\rS,F_{\rB_\rS},\eta),$ define the $\Theta_\lambda$-Slodowy functor $F_{\Theta_\lambda}:\Bb_{\fs,X}^\lambda(\rG)\to\Op_{X}^\lambda(\rG,\rP)$ by 
 \[F_{\Theta_\lambda}(\Xi_\lambda)=(\lambda~,~E_{\rC}\star F_{\rS}(\rG)~,~E_\rC\star F_{\rB_\rS}(\rP)~,~\eta\star\psi_0+\psi_{m_1}+\cdots+\psi_{m_N})~,\]
 where $\Xi_\lambda=(\lambda,E_\rC,\psi_0,\psi_{m_1},\cdots,\psi_{m_N})$. 

 Lemmas \ref{lemma local no indices} and \ref{lem local statement for functor} are valid when the Maurer-Cartan form $\theta_\rP$ in \eqref{eq: local difference lambda=1} and \eqref{preserves connection} are replaced by $\lambda\theta_\rP$. Moreover, the natural $\lambda$-connection generalization of Theorem \ref{thm: parameterization picking PSL2 oper} is proven by replacing the use of Lemma \ref{lem local statement for functor} by the $\lambda$-version. In particular, this implies that the $\lambda$-Slodowy functor $F_{\Theta_\lambda}:\Bb_{\fs,X}^\lambda(\rG)\to\Op_{X}^\lambda(\rG,\rP)$ is an equivalence of categories. 
Moreover, each section $\tau:\C\to\Qq\Op_{X}(\rS,\rB_\rS)$ defines an equivalence of categories
    \[F_\tau:\xymatrix@R=0em{\Qq\Bb_{\fs, X}(\rG)\ar[r]&\Qq\Op_{X}(\rG,\rP) \\\Xi_\lambda\ar@{|->}[r]&F_{\tau(\lambda)}(\Xi_\lambda)} \]

\begin{remark}\label{rem sections and hitchin section}
Recall that one such sections $\tau:\C\to\Qq\Op_{X}(\rS,\rB_\rS)$ is defined in \eqref{eq sections of lambda}. When $\langle f,h,e\rangle$ is a principal $\fsl_2$-triple, the functor 
\begin{equation}
    \label{eq: Slodowy functor recovering Hitchin section}F_{\tau(0)}:\Bb_{\fs,X}^0(\rG)\to\Op_{X}^0(\rG,\rB)\hookrightarrow \Ff_{X}^0(\rG)
\end{equation} recovers the Hitchin section from \eqref{eq: Hitchin section}.
\end{remark}

For the group $\rSL_4\C$, explicit models of the $(\lambda,\fs)$-Slodowy category are given by replacing the $\rC$-connection in \S \ref{sec: explicit models} with a $(\lambda,\rC)$-connection. Explicit descriptions of the $\lambda$-Slodowy functor are defined by replacing the $\rSL_2\C$-oper from \eqref{eq psl2 oper in smooth splitting} with the $(\lambda,\rSL_2\C)$-oper $\tau(\lambda)$ defined above. In particular, the associated Higgs bundles are obtained by setting $\lambda=0.$

Finally, analogous to Theorem \ref{thm: no quadratic equivalence}, the above equivalence can be upgraded to remove choice of section $\tau.$ 
There are natural $\C^*$-actions on $\Qq\Bb_{\fs,X}(\rG)$ and $\Qq\Op_{X}(\rG,\rP) $ defined by 
\begin{equation}
     \label{eq C*actions}
     \xymatrix@R=0em{\xi\cdot(\lambda, E_\rG,E_\rP,\omega)=(\xi\lambda,E_\rG,E_\rP,\xi\omega)\\
    \xi\cdot(\lambda,E_\rC,\psi_0,\psi_{m_1},\cdots,\psi_{m_N})=(\xi\lambda,E_\rC,\xi\psi_0,\xi\psi_{m_1},\cdots,\xi\psi_{m_N})}
 \end{equation} 
Denote the fiber product of the categories $\Qq\Op_{X}(\rS,\rB_\rS)$ and $\Qq\widehat\Bb_{\fs, X}(\rG)$ with respect to the natural projections to $\C$ by
\[\Qq\Op_{X}(\rS,\rB_\rS)\times_\C \Qq\widehat\Bb_{\fs, X}(\rG).\]
The diagonal $\C^*$-action on $\Qq\Op_{X}(\rS,\rB_\rS)\times\Qq\widehat\Bb_{\fs, X}(\rG)$ induces a natural $\C^*$-action on the fiber product. 

 The proof of the following theorem is almost identical to the proof of Theorem \ref{thm: no quadratic equivalence}. 
\begin{theorem}\label{Thm lambda equiv}
Let $\rG$ be a connected complex semisimple Lie group, $\fs=\langle f,h,e\rangle\subset\fg$ be an even $\fsl_2$-triple, $\rS<\rG$ the connected subgroup with Lie algebra $\fs$ and $\rP<\rG$ be the associated even JM-parabolic. Let $F_{\Theta_\lambda}$ be the $\Theta_\lambda$-Slodowy functor associated to an $(\lambda,\rS,\rB_\rS)$-oper $\Theta_\lambda$. Then, the functor 
    \[\Qq F:\Qq\Op_{X}(\rS,\rB_\rS)\times_\C \Qq\widehat\Bb_{\fs, X}(\rG)\to \Qq\Op_{X}(\rG,\rP)\]
    defined by $F(\Theta_\lambda,\hat\Xi_\lambda)=F_{\Theta_\lambda}(\hat\Xi_\lambda)$ is an equivalence of categories when $\rS\cong\rPSL_2\C$ and essentially surjective and full when $\rS\cong\rSL_2\C$. Moreover, $\Qq F$ is equivariant with respect to the natural $\C^*$-actions from \eqref{eq C*actions}.
\end{theorem}


\section{More explicit examples and related objects}
\label{Section examples}
So far all of our explicit examples have been for $\rG=\rSL_4\C$ or the parabolic being the Borel subgroup. In this section we discuss a few more explicit examples which generalize the $\rSL_4\C$-examples discussed throughout the paper. We also discuss the relation between some of the examples below and so called higher Teichm\"uller spaces.

The parabolics of the groups $\rSO_n\C$ and $\rSp_{2m}\C$ will be described in terms of the subspaces they stabilize in $\C^n$ ($n=2m$ for $\rSp_{2m}\C$) equipped with a nondegenerate bilinear form which is symmetric for $\rSO_n\C$ and skew-symmetric for $\rSp_{2m}\C$. Also, a holomorphic vector bundle $\Vv$ with a holomorphic volume form $\Omega$ will sometimes be equipped with a compatible holomorphic symplectic structure $B$ or a compatible holomorphic orthogonal structure $Q.$ By this we mean, 
\begin{itemize}
	\item$B \in \rH^0(X, \Lambda^2\Vv^{*})$ is non-degenerate and compatible with $\Omega.$
	\item$Q \in \rH^0(X, \Sym^2\Vv^{*})$ is non-degenerate and compatible with $\Omega.$
\end{itemize}


\subsection{$\rB$-opers for classical groups}
An $(\rSL_n\C,\rB,\lambda)$-oper is a tuple 
\[(\lambda,\Vv,\Vv_1\subset\Vv_2\subset\cdots\subset\Vv_n,\nabla_\Vv),\]
where $\lambda\in\C$, $\Vv$ is a rank $n$ holomorphic vector bundle with $\det\Vv\cong\Oo_X,$ $\Vv_i$ is a rank $i$ holomorphic sub-bundle, and $\nabla_\Vv$ is a $\lambda$-connection on $\Vv$ compatible with the isomorphism $\det\Vv\cong\Oo_X,$ such that $\nabla(\Vv_i)\subset \Vv_{i+1}\otimes\Kk$ and $\nabla_\Vv$ induces an isomorphism $\Vv_i/\Vv_{i-1}\to\Vv_{i+1}/\Vv_i\otimes \Kk.$ 

For $\rSO_{2n+1}\C$ (resp. $\rSp_{2n}\C$), a $(\rB,\lambda)$-oper is an $(\rSL_{2n+1}\C,\rB,\lambda)$-oper (resp. $(\rSL_{2n}\C,\rB,\lambda)$-oper), where $\Vv$ is equipped with a holomorphic orthogonal (resp. symplectic) structure of volume $1$ which is preserved by $\nabla_\Vv$, and with respect to which $\Vv_i$ is isotropic and $\Vv_{n-i}=\Vv_i^\perp$ for $1\leq i\leq n$. 
For $\rG=\rSO_{2n}\C$, $(\rB,\lambda)$-opers are described in \S \ref{sec: more ex son} when $2n=2k+2.$ 

An object in the category $\Qq\Bb_{\fs, X}(\rSL_n\C)$ is a tuple $(\lambda,\Ll,\psi_1,\cdots,\psi_{n-1})$, where $\lambda\in\C$, $\Ll$ is a holomorphic line bundle with $\Ll^n\cong\Oo_X$ and $\psi_j\in \rH^0(X,\Kk^{j+1}).$ For the groups $\rSO_{2n+1}\C$ and $\rSp_{2n}\C$ one has $\Ll\cong\Oo_X$ and $\Ll^2\cong\Oo_X$, respectively, and $\psi_{2j}=0$ in both cases. 

For the section $\tau$ of $\Qq\Op_{X}(\rPSL_2\C)\rightarrow \C$ from \eqref{eq sections of lambda}, the Slodowy functor $F_{\tau}:\Qq\Bb_{\fs,X}(\rSL_n\C)\to\Qq\Op_{X}(\rSL_n\C,\rB)$ is given by
	 \[F_{\tau}(\lambda,\Ll,\psi_1,\cdots,\psi_{n-1})=(\lambda,\bar\partial_\Ll\otimes\bar\partial_\Vv,\nabla_\Ll\otimes\nabla_\Vv),
    \]
    where, $\nabla_\Ll$ is the holomorphic $\lambda$-connection inducing the trivial connection on $\Ll^n=\Oo_X,$  and $(\bar\partial_\Ll\otimes\bar\partial_\Vv,\nabla_\Ll\otimes\nabla_\Vv)$ are Dolbeault operators and $\lambda$-connections on the smooth bundle $\Kk^{\frac{n-1}{2}}\oplus \Kk^{\frac{n-3}{2}}\oplus\cdots\oplus\Kk^{\frac{1-n}{2}}$ given by
    \begin{equation}
    	\label{eq: SLn B-oper explicit}
    	\bar\partial_\Vv=\left(\begin{smallmatrix}
        \bar\partial_\frac{n-1}{2} &\lambda\mu_1h&\\&\bar\partial_\frac{n-3}{2}&\lambda\mu_2h\\&&\ddots&\ddots\\&&&\bar\partial_\frac{3-n}{2}&\lambda\mu_{n-1}h\\&&&&\bar\partial_\frac{1-n}{2}\end{smallmatrix}\right)
    \end{equation}
     \[\nabla_\Vv=\left(\begin{smallmatrix}
        \lambda\nabla^h_\frac{n-1}{2}&\mu_1( q+\psi_1)&\psi_2&\dots&\psi_{n-1}\\1&\lambda\nabla_\frac{n-3}{2}^h&\mu_2(q+\psi_1)&\dots&\psi_{n-2}\\ &\ddots&\ddots&\ddots&\vdots\\\\&&1&\lambda\nabla^h_\frac{3-n}{2}&\mu_{n-1}(q+\psi_{1})\\&&&1&\lambda\nabla^h_\frac{1-n}{2}\end{smallmatrix}\right).\]
        Here $\mu_i=i(n-i)$ and the coefficient on each $\psi_i$ are 1 for $i>1.$ The $\rSO_{2n+1}\C$ (resp. $\rSp_{2n}\C$) Slodowy functor is given by the restriction of the $\rSL_{2n+1}\C$ (resp. $\rSL_{2n}\C$) Slodowy functor. Note that $\Vv$ has natural orthogonal (resp. symplectic) structure when $\rk(\Vv)=2n+1$ (resp. $\rk(\Vv)=2n$) which is preserved by $\nabla_\Vv$ when $\psi_{2j}=0$ for all $j.$

\subsection{Compact dual of Tube-type Hermitian symmetric space}\label{section: max}
A family of even JM-parabolics $\rP<\rG$ arise when $\rG/\rP$ is the compact dual of a non-compact Hermitian symmetric space of tube type.  There are five such pairs $(\rG,\rP)$:
\begin{enumerate}
    \item $\rG=\rSL_{2n}\C$ and $\rP$ the stabilizer of a $\C^n\subset\C^{2n}$;
    \item $\rG=\rSp_{2n}\C$ and $\rP$ the stabilizer of an \emph{isotropic} subspace $\C^n\subset\C^{2n}$;
    \item $\rG=\rSO_{4n}\C$ and $\rP$ the stabilizer of an \emph{isotropic} subspace $\C^{2n}\subset\C^{4n};$
    \item $\rG=\rSO_n\C$ and $\rP$ the stabilizer of an \emph{isotropic} line $\C\subset\C^n;$
    \item $\rG=\rE_7$ and, if $\sum_{i=1}^7n_i\alpha_i$ is an expression of the longest root with respect to a choice of simple roots, $\rP$ is the maximal parabolic subgroup associated to the unique simple root $\alpha_k$ with $n_k=1.$
\end{enumerate}

Each such parabolic is a maximal parabolic which corresponds to a simple root $\alpha_k$ with $n_k=1$, where the longest root is given by $\sum_{i=1}^{\rk(\fg)}n_i\alpha_i.$ Thus, the $\Z$-gradings \eqref{eq associated graded Z grading of parabolic} are $\fg=\fg_{-1}\oplus\fg_0\oplus\fg_{1}$ and the $\fsl_2$-module decompositions \eqref{eq sl2 module decomp} are $\fg=W_0\oplus W_2.$
The centralizers $\rC$ and multiplicities $(n_0,n_2)$ are 
\begin{center}
    \begin{tabular}{|c|c|c|c|c|c|}
\hline
$\rG$& $\rSL_{2n}\C$&$\rSp_{2n}\C$ &$\rSO_{4n}\C$& $\rSO_n\C$&$\rE_7$\\
\hline
$\rC$&$\rSL^\pm_n\C$&$\rO_n\C$&$\rSp_{2n}\C$&$\rO_{n-3}\C$&$\rF_4$\\
\hline
$n_0$& $n^2-1$&$\frac{n(n-1)}{2}$&$n(2n+1)$&$\frac{(n-2)(n-3)}{2}$&$52$\\
\hline
$n_2$& $n^2$&$\frac{n(n+1)}{2}$&$n(2n-1)$&$n-2$&$27$\\
\hline
\end{tabular}
\end{center}

We now describe the vector bundle definition of an $(\rG,\rP,\lambda)$-oper in cases (1)-(3), case (4) will be described in the next subsection. 

An $(\rSL_{2n}\C,\rP,\lambda)$-oper consists of a tuple $(\lambda,\Vv,\Vv_n,\nabla_\Vv)$, where $\lambda\in\C$, $\Vv$ is a rank $2n$ holomorphic vector bundle with $\Lambda^{2n}V\cong\Oo_X$, $\nabla_\Vv$ is a $\lambda$-connection on $\Vv$  preserving $\Lambda^{2n}V\cong\Oo_X$, and $\Vv_n\subset\Vv$ is a holomorphic rank $n$ sub-bundle such that the induced map $\nabla_\Vv:\Vv_n\to\Vv/\Vv_n\otimes\Kk$ is an isomorphism. 

For case (2), an $(\rSp_{2n}\C,\rP,\lambda)$-oper is an $(\rSL_{2n},\C,\rP,\lambda)$-oper $(\lambda,\Vv,\Vv_n,\nabla_\Vv)$ such that $\Vv$ is equipped holomorphic symplectic structure $B_\Vv$  which is preserved by $\nabla_\Vv$ and with respect to which $\Vv_n$ is isotropic. 

For case (3), an $(\rSO_{4n}\C,\rP,\lambda)$-oper is an $(\rSL_{4n},\C,\rP,\lambda)$-oper $(\lambda,\Vv,\Vv_{2n},\nabla_\Vv)$ such that $\Vv$ is equipped holomorphic orthogonal structure $Q_\Vv$ which is preserved by $\nabla_\Vv$ and with respect to which $\Vv_{2n}$ is isotropic. 

For case (1), the objects $\Qq\Bb_{\fs,X}(\rSL_{2n}\C)$ consist of tuples $(\lambda,\Ww,\nabla_\Ww,\psi_1)$, where $\lambda\in\C$, $\Ww$ is a rank $n$ holomorphic vector bundle with $\det(\Ww)^2\cong\Oo_X,$ $\nabla_\Ww$ is a holomorphic $\lambda$-connection on $\Ww$ and $\psi_1\in \rH^0(X,\End(\Ww)\otimes \Kk^2).$
For the section $\tau$ of $\Qq\Op_{X}(\rPSL_2\C)$ from \eqref{eq sections of lambda}, the Slodowy functor is
	 \[F_{\tau}(\lambda,\Ww,\nabla_\Ww,\psi_1)=(\lambda,\bar\partial_\Vv,\nabla_\Vv),\]
    where $(\bar\partial_\Vv,\nabla_\Vv)$ are Dolbeault operators and $\lambda$-connections on the smooth bundle $\Ww\otimes \Kk^\frac{1}{2}\oplus \Ww\otimes \Kk^{-\frac{1}{2}}$ given by
   \begin{equation}
   	\label{eq: dolbeault sl2n}\bar\partial_\Vv=\begin{pmatrix}
        \bar\partial_{\Ww}\otimes \bar\partial_{1/2}&\lambda h\otimes \Id_\Ww\\0&\bar\partial_{\Ww}\otimes \bar\partial_{-1/2}\end{pmatrix}\ \ ,\ \  \nabla_\Vv=\begin{pmatrix}
            \nabla_\Ww\otimes \lambda \nabla_{1/2}^h& q\otimes\Id_\Ww+\psi_1\\\Id_\Ww&\nabla_\Ww\otimes\lambda\nabla_{-1/2}^h
        \end{pmatrix}.
 \end{equation}

 For case (2), the objects of $\Qq\Bb_{\fs,X}(\rSp_{2n}\C)$ are objects $(\lambda,\Ww,\nabla_\Ww,\psi_1)$ of $\Qq\Bb_{\fs,X}(\rSL_{2n}\C)$ where $\Ww$ is equipped with a holomorphic orthogonal structure $Q_\Ww$ which is preserved by $\nabla_\Ww$ and with respect to which $\psi_1$ is symmetric. The Slodowy functor $F_{\tau}$ is the restriction of functor on $\Qq\Bb_{\fs,X}(\rSL_{2n}\C)$. The symplectic structure on $\Ww\otimes K^\frac{1}{2}\oplus \Ww\otimes K^{-\frac{1}{2}}$ is defined by 
 \[B_\Vv=\begin{pmatrix}
 	0&Q_\Ww\\-Q_\Ww&0
 \end{pmatrix}.\]

For case (3), the objects of $\Qq\Bb_{\fs,X}(\rSO_{4n}\C)$ are objects $(\lambda,\Ww,\nabla_\Ww,\psi_1)$ of $\Qq\Bb_{\fs,X}(\rSL_{4n}\C)$ where $\Ww$ is equipped with a holomorphic symplectic structure $B_\Ww$ which is preserved by $\nabla_\Ww$ and with respect to which $\psi_1$ is symmetric. The Slodowy functor $F_{\tau}$ is the restriction of functor on $\Qq\Bb_{\fs,X}(\rSL_{4n}\C)$. The orthogonal structure on $\Ww\otimes K^\frac{1}{2}\oplus \Ww\otimes K^{-\frac{1}{2}}$ is defined by 
 \[Q_\Vv=\begin{pmatrix}
 	0&B_\Ww\\-B_\Ww&0
 \end{pmatrix}.\]

\subsection{Some more opers for $\rSL_n\C$ and $\rSO_n\C$}\label{sec: more ex son}
Another family of even JM-parabolics of $\rSL_n\C$ for are given by the stabilizers of partial flags
\[\C^1\subset\C^2\subset\cdots\subset\C^k\subset \C^{n-k}\subset\cdots\subset\C^{n-1}\subset\C^n\]
where $2k<n.$ When $\C^n$ is equipped with an orthogonal structure and the subspaces $\C^i$ are isotropic for $1\leq i\leq k$, the associated parabolic $\rP<\rSO_n\C$ is also an even JM-parabolic.
In particular, for the orthogonal group, the parabolic $\rP$ is the Borel subgroup when $n=2k+1$ and $n=2k+2$, and $\rP$ is the parabolic from case (4) of the previous section when $k=1$.

For the above parabolic $\rP,$ an $(\rSL_n\C,\rP,\lambda)$-oper consists of a tuple 
\[(\lambda~,~\Vv~,~\Vv_1\subset\cdots\Vv_{k}\subset\Vv_{n-k}\subset\cdots\subset\Vv~,~\nabla_\Vv),\]
where $\lambda\in\C$, $\Vv$ is a rank $n$ holomorphic vector bundle with $\Lambda^{n}V\cong\Oo_X$, $\nabla_\Vv$ is a $\lambda$-connection on $\Vv$ compatible with $\Lambda^nV\cong\Oo_X$ and $\Vv_i\subset\Vv$ is a holomorphic rank $i$ sub-bundle such that:
\begin{itemize}
	\item for $i\neq k$, $\nabla_\Vv(\Vv_i)\subset\Vv_{i+1}\otimes\Kk$ and $\nabla_\Vv$ induces an isomorphism $\Vv_{i}/\Vv_{i-1}\cong\Vv_{i+1}/\Vv_{i}\otimes \Kk$;
	\item $\nabla_\Vv(\Vv_{k})\subset\Vv_{n-k}\otimes\Kk$, and $\nabla^2_\Vv$ induces an isomorphism $\Vv_{k}/\Vv_{k-1}\cong\Vv_{n-k+1}/\Vv_{n-k}\otimes \Kk^2$.
\end{itemize}
When $\rG=\rSO_n\C$, an $(\rSO_n\C,\rP,\lambda)$-oper is an $(\rSL_n\C,\rP,\lambda)$-oper where the holomorphic bundle $\Vv$ is equipped with an orthogonal structure $Q_\Vv$ which is preserved by $\nabla_\Vv$ and with respect to which $\Vv_i$ is isotropic and $\Vv_{n-i}=\Vv_i^{\perp_{Q_\Vv}}$ for $1\leq i\leq k$.

We will describe the $\fs$-Slodowy category and functor for $\rG=\rSO_n\C$, the $\rG=\rSL_n\C$ case is left to the reader.  
The $\fsl_2$-module decomposition of the even JM-parabolic is 
\[\fso_n\C=W_0\oplus W_{2k}\oplus\bigoplus\limits_{\substack{j=1\\j\neq\frac{k}{2}+1}}^{k}W_{4j-2}.\]
The centralizer $\rC$ of the even $\fsl_2$ is isomorphic to $\rO_{n-2k+1}\C,$ and the multiplicities are $n_0=\frac{(n-2k-1)(n-2k-2)}{2},$ $n_{4j-2}=1$ for $4j-2\neq 2k,$ and
\[n_{2k}=\begin{dcases}
	n-2k&\text{if $k$ odd}\\
	n-2k-1&\text{if $k$ even}
\end{dcases}.\]
When $k$ is odd, the $\rC$-representation space $Z_{2k}$ decomposes as a direct sum of the one dimensional trivial representation and the standard $(n-2k-1)$-dimensional representation twisted by the determinant representation.

The objects of the category $\Qq\Bb_{\fs,X}(\rSO_n\C)$ consists of tuples 
\[(\lambda,\Ww,\nabla_\Ww, \psi_1,\cdots, \psi_{2k-1}),\]
where $\lambda\in\C$, $\Ww$ is a rank $(n-2k-1)$ holomorphic vector bundle equipped with an orthogonal structure $Q_\Ww$, $\nabla_\Ww$ is a $\lambda$-connection on $\Ww$ which preserves $Q_\Ww$, $\psi_{2j-1}\in \rH^0(X,\Kk^{2j})$ for $2j-1\neq k$ and  
\[\psi_{k}=\begin{dcases}
	(q_{k},\hat\psi_k)\in \rH^0(X,K^{k+1})\oplus \rH^0(X,\Ww\otimes\det(\Ww)\otimes \Kk^{k+1})&\text{if $k$ odd}\\
	\hat\psi_k\in \rH^0(X,\Ww\otimes\det(\Ww)\otimes \Kk^{k+1})&\text{if $k$ even}.
\end{dcases}\]

The Slodowy functor $F_{\tau}:\Qq\Bb_{\fs,X}(\rSO_n\C)\to\Qq\Op_{X}(\rSO_n\C,\rP)$ is given by 
\[F_{\tau}(\lambda,\Ww,\nabla_\Ww, \psi_1,\cdots, \psi_{2k-1})=(\lambda,\bar\partial_\Vv,\nabla_V),\]
where $(\bar\partial_\Vv,\nabla_\Vv)$ are Dolbeault operators and $\lambda$-connections on the bundle $\Vv=\Ww\oplus \Ww_{2k+1}\otimes\det\Ww$ given by
\begin{equation}
   	\label{eq: dolbeault son}\bar\partial_\Vv=\begin{pmatrix}
        \bar\partial_{\Ww}&0\\0& \bar\partial_{\Ww_{2k+1}}\otimes\bar\partial_{\det \Ww}\end{pmatrix}\ \ ,\ \  \nabla_\lambda=\begin{pmatrix}
            \nabla_\Ww &\hat\Psi_k^T\\\hat\Psi_k&\nabla_{\Ww_{2k+1}}\otimes \nabla_{\det \Ww}
        \end{pmatrix}.
 \end{equation}
Here, 
\begin{itemize}
	\item $(\lambda,\Ww_{2k+1},\nabla_{\Ww_{2k+1}})$ denotes the $(\rSO_{2k+1}\C,\rB,\lambda)$-oper \eqref{eq: SLn B-oper explicit} associated to the Slodowy functor evaluated on 
\[\begin{dcases}
	(\psi_1,\psi_3,\psi_{k-2},q_{k},\psi_{k+2},\cdots,\psi_{2k})&\text{ if $k$ is odd}\\
	(\psi_1,\psi_3,\cdots,\psi_{2k})&\text{if $k$ is even};
\end{dcases}\]
\item   $\hat\Psi_k:\Ww_{2k+1}\otimes\det\Ww\to \Ww\otimes\Kk$ is a holomorphic bundle map which, in the smooth splitting of $\Ww_{2k+1}$ from \eqref{eq: SLn B-oper explicit}, is given by 
\[\hat\Psi_k=\begin{pmatrix}
	0&\cdots&0&\hat\psi_k
\end{pmatrix}:(\Kk^{k}\oplus\cdots\oplus \Kk^{-k})\otimes\det\Ww\to\Ww\otimes\Kk.\]
\end{itemize}
The orthogonal structure $Q_\Vv$ on $\Ww\oplus\Ww_{2k+1}\otimes\det\Ww$ is given by
\[Q_\Vv=\begin{pmatrix}
	Q_\Ww&0\\0&Q_{\Ww_{2k+1}}\otimes \det Q_\Ww
\end{pmatrix}.\]

\subsection{Remarks on Nonabelian Hodge and Higher Teichm\"uller spaces}
So far, we have avoided the discussion of moduli spaces and stability. However, for every $\lambda$, the categories $\Bb_{\fs,X}^\lambda(\rG)$ and $\Op_{X}^\lambda(\rG,\rP)\subset \Ff_{X}^\lambda(\rG)$ have natural stability conditions and corresponding coarse moduli spaces. 

We expect the natural stability conditions on $\Bb_{\fs,X}^\lambda(\rG)$ and $\Ff_{X}^\lambda(\rG)$ to be compatible so that the Slodowy functor $F_{\Theta_\lambda}$ induces a well defined map on coarse moduli spaces. 
This is especially desirable when $\lambda=0$ because one can associate a holomorphic connection to a point in the moduli space of poly-stable Higgs bundles via the non-abelian Hodge correspondence. 

 Denote the moduli spaces of isomorphism classes of poly-stable $\rG$-Higgs bundles and reductive holomorphic connections on a compact Riemann surface $X$ by $\Mm_{X}^0(\rG)$ and $\Mm_{X}^1(\rG)$ respectively. We refer the reader to \cite{SimpsonModuli1,SimpsonModuli2} for the construction of these coarse moduli spaces.
 
 Furthermore, the moduli space $\Mm_{X}^1(\rG)$ is analytically isomorphic to the complex analytic variety (the \emph{character variety}) $\Xx_{X}(\rG)$ of conjugacy classes of reductive homomorhisms of the fundamental group of $X$ in $\rG$:
 \[\Xx_{X}(\rG)= \Hom^{\textnormal{red}}(\pi_1(X),\rG)/ \rG.\]
 Note that $\Xx_{X}(\rG)$ as a complex analytic variety only depends on the topological surface underlying the Riemann surface $X.$

 The non-abelian Hodge correspondence (see \cite{selfduality,canonicalmetrics,harmoicmetric,SimpsonVHS}) defines a real analytic isomorphism between these spaces:
\[\Tt:\Mm_{X}^0(\rG)\to\Mm_{X}^1(\rG)\cong\Xx_{X}(\rG).\]

For $(0,\rPSL_2\C,\rB)$-opers, the Higgs bundles are always stable and the image under the map $\Tt$ can be identified with the Teichm\"uller space of the underlying topological surface \cite{selfduality}. More precisely, the image of $(0,\rPSL_2\C,\rB)$-opers under the non-abelian Hodge correspondence $\Tt$ consists of all conjugacy classes of holonomies of hyperbolic structures on the underlying topological surface.

For $(0,\rG,\rB)$-opers (i.e., the Hitchin section) it is not hard to show that the Higgs bundles are always stable and so define points in $\Mm_{X}^0(\rG).$ Moreover, applying non-abelian Hodge to this locus defines a (union of) connected component of the character variety of representations into the split real form of $\rG$ called the Hitchin component \cite{liegroupsteichmuller}.  Moreover, the surface group representations in this component generalize many features of Teichm\"uller space. In particular, they are all discrete and faithful quasi-isometric embeddings \cite{AnosovFlowsLabourie,fock_goncharov_2006}, and are holonomies of locally homogeneous  geometric structures on closed manifolds \cite{guichard_wienhard_2012,KLPAnosov1}. The Hitchin component was the first example of what is now referred to as a higher Teichm\"uller space/component.

The family of $(0,\rG,\rP)$-opers described in \S \ref{section: max} are related to a family of higher Teichm\"uller spaces known as maximal representations into a real Hermitian Lie group of tube type. In these cases, when the Higgs field $\psi_0$ is identically zero and $q=0$, the Slodowy functor $F_{\tau(0)}$ recovers the Cayley correspondence of \cite[\S 5]{BGRmaximalToledo} (see also \cite{sp4GothenConnComp,HermitianTypeHiggsBGG,UpqHiggs}). This correspondence relates Higgs bundles with so called maximal Toledo invariant for a real\footnote{We have not discussed Higgs bundles for real groups, for appropriate definitions see for example \cite{GothenSurvey}.} Hermitian Lie group of tube type $\rG^\R$ with $\Kk^2$-twisted Higgs bundles for another group. 
For the complex groups $\rSL_{2n}\C,~\rSp_{2n}\C,~\rSO_{4n}\C,~\rSO_n\C~$ and $\rE_7$, the Hermitian Lie group $\rG^\R$ is $\rSU_{n,n},$ $\rSp_{2n}\R,$ $\rSO^*_{4n},$ $\rSO_{2,n-2}$ and $\rE_7^{-25}$, respectively. 

Under the non-abelian Hodge correspondence, the isomorphism classes of such Higgs bundles which are poly-stable are in bijective correspondence with the set of conjugacy classes of so called maximal representations. Such representations have been studied by many authors, and correspond to unions of connected components of the character variety which consist entirely of discrete and faithful representations \cite{BIWmaximalToledoAnnals}.

In contrast, the  $(0,\rSL_n\C,\rP)$-opers from \S \ref{sec: more ex son} are \emph{not} related to unions of connected components of the character variety of a real form of $\rSL_n\C.$ However, the $(0,\rSO_n\C,\rP)$-opers from \S \ref{sec: more ex son} \emph{are} related to connected components of character varieties. In this case, when the Higgs field $\psi_0$ is identically zero and $q=0$, the Slodowy functor $F_{\tau(0)}$ recovers the generalized Cayley correspondence of \cite{so(pq)BCGGO}. 
Similar to the case of maximal representations, under the non-abelian Hodge correspondence, the isomorphism classes of such Higgs bundles which are poly-stable are identified with unions of connected components of the character variety for the real form $\rSO(k-1,n-k+1)$. 




\begin{remark}
For a general even JM-parabolic, when $\lambda=0$ and $\psi_0=0,$ the Slodowy functor produces poly-stable Higgs bundles whose associated representation via the non-abelian Hodge correspondence is \emph{not} always valued in a real form. However, using Guichard-Wienhard's work on $\Theta$-positivity \cite{PosRepsGWPROCEEDINGS}, one expects one more such family for certain real forms of the exceptional groups $\rF_4,$ $\rE_6,$ $\rE_7$ and $\rE_8.$ This is indeed the case and the relationship between the Slodowy functor and higher Teichm\"uller spaces will be described in \cite{MagicalBCGGO}. 
\end{remark}


\bibliography{Notesbib}{}
\bibliographystyle{plain}

\end{document}